\documentclass{amsart}

\usepackage[all]{xy}

\usepackage{mathrsfs,graphicx,latexsym,tikz,color,euscript}
\usepackage{amsfonts,amssymb,amsmath,amscd,amsthm, enumerate}
\usepackage{graphicx}
\usepackage{epstopdf}
\usepackage{hyperref}
\usepackage{upgreek}
\usepackage{verbatim}

\def\R{\mathbb R}

\def\C{\mathbb C}
\def \E {\mathbb E}
\def \V {\mathbb V}

\def\N {\mathbb N}

\def\w{\omega}

\def \nbhd {neighborhood }

\def\into{\rightarrow}

\def\metric#1{\langle #1 \rangle}
\def\metrictwo#1{\langle\!\langle #1 \rangle\!\rangle}
\def\normtwo#1{\| #1 \|}
\def\m#1{\begin{bmatrix}#1\end{bmatrix}}

\newcommand{\al}{\alpha}

\newcommand{\g}{\gamma}
\newcommand{\s}{\sigma}

\newcommand{\tht}{\theta}
\newcommand{\lmd}{\lambda}

\newcommand{\Eh}{\hat{\mathbb{E}}}

\newcommand{\Sh}{\hat{S}}
\newcommand{\CE}{\mathcal{E}}
\newcommand{\CEh}{\hat{\mathcal{E}}}
\newcommand{\CH}{\mathcal{H}}
\newcommand{\CU}{\mathcal{U}}
\newcommand{\CV}{\mathcal{V}}
\newcommand{\CHS}{\mathcal{HS}}
\newcommand{\CS}{\mathcal{S}}

\newcommand{\nbtl}{\tilde{\nabla}}

\newtheorem{theorem}{Theorem}[section]
\newtheorem*{theorem*}{Theorem}
\newtheorem{proposition}{Proposition}[section]

\newtheorem{lemma}{Lemma}[section]
\newtheorem{definition}{Definition}[section]
\newtheorem{remark}{Remark}[section]
\newtheorem{example}{Example}[section]

\newtheorem{problem}{Problem}

\DeclareMathOperator{\diag}{diag}
\DeclareMathOperator{\sech}{sech}

\begin{document}

\title[Scattering  $n$ bodies]{Chazy-Type Asymptotics and Hyperbolic Scattering for the $n$-Body Problem}

\author[N. Duignan]{Nathan Duignan} 
\address{School of Mathematics and Statistics, University of Sydney, Camperdown, 2006 NSW, Australia}
\email{nathan.duignan@sydney.edu.au}

\author[R. Moeckel]{Richard Moeckel}
\address{School of Mathematics\\ University of Minnesota\\ Minneapolis MN 55455}
\email{rick@math.umn.edu}

\author[R. Montgomery]{Richard Montgomery}
\address{Mathematics Department\\ University of California, Santa Cruz\\
Santa Cruz CA 95064}
\email{rmont@ucsc.edu}

\author[G. Yu]{Guowei Yu} 
\address{Chern Institute of Mathematics and LPMC\\ Nankai University\\ Tianjin, China, 300071}
\email{yugw@nankai.edu.cn}


\begin{abstract}  
We study solutions of the Newtonian $n$-body problem which tend to infinity hyperbolically, that is, all mutual distances tend to infinity with nonzero speed as $t\into +\infty$ or as $t\into -\infty$.  In suitable coordinates, such solutions form the stable or unstable manifolds of normally hyperbolic equilibrium points in a boundary manifold ``at infinity''.   We show that the flow near these manifolds can be analytically linearized and use this to give a new proof of Chazy's classical asymptotic formulas.  We also address the scattering problem, namely, for solutions which are hyperbolic in both forward and backward time, how are the limiting equilibrium points related?  After proving some basic theorems about this scattering relation, we use perturbations of our manifold at infinity to study scattering ``near infinity'', that is, when the bodies stay far apart and interact only weakly. 
\end{abstract}

\maketitle

\section{Introduction and Overview} 

Let $q(t)=(q_1(t),q_2(t),\ldots, q_n(t))\in \E,\, t\in\R$ be a solution of the
Newtonian $n$-body problem. Here $\E$ denotes the n-body configuration space
with center of mass fixed at the origin.  If $q(t)$ has positive energy $h$,  then it is unbounded: some of its interbody distances will  tend  to infinity with time.  Following Chazy \cite{Chazy},  we call   a solution  ``hyperbolic''  if {\it all}   its interbody distances tend to infinity with time, and do so asymptotically linearly. Chazy \cite[eq.~(27)]{Chazy} established the first few terms of a convergent asymptotic expansion for  hyperbolic solutions, 
$$q(t) = At + B\log |t| + C +\cdots.$$ 
One can show that $B= B(A)$.  The coefficients $A, C\in\E$ determine the entire series and  can be thought of as initial conditions at infinity.  We will call these two coefficients  the scattering parameters.  The leading coefficient $A$ represents  the asymptotic velocities of the bodies. The  mass weighted norm of $A$ 
satisfies $\normtwo{A}^2=2h$.  The normalized vector $A/\normtwo{A}$ represents the limiting shape of the configuration $q(t)$.  The interpretation of the coefficient $C$ is less straightforward.  
  
Bi-hyperbolic solutions, solutions that are hyperbolic in both time directions, sweep out an open subset of the phase space and allow us to define a  {\em scattering map}  
$$F: (A, C) \mapsto (A', C')$$
which sends the  scattering parameters associated to the infinite past  of  a solution to its  scattering parameters associated to   the  infinite future.
We then say that    $A$ and $A'$ are related by hyperbolic scattering and write $A\rightarrow A'$.  In other words, $A\rightarrow A'$ means that there is some bi-hyperbolic orbit with asymptotic behavior described by $A$ in the past and by $A'$ in the future. 
{\it The main problem of hyperbolic scattering can now be stated :  For a given scattering parameter $A$ at time minus infinity, which parameters $A'$ at time infinity are related to it  by hyperbolic scattering ?}   Since the energy is constant along solutions,   we must have $\normtwo{A} =\normtwo{A'}$, so 
this question  is really a question of connecting the  two asymptotic shapes $A/ \|A\|$ and $A' /\| A' \|$ .  

For example, consider the planar two-body problem with the center of mass $m_1 q_1 + m_2 q_2$ fixed at the origin.  Then the vector ${\vec r} (t) = q_2(t)-q_1(t)$ connecting the two-bodies satisfies the Kepler problem and for positive energies, it sweeps out a hyperbola in the plane (see Figure~\ref{fig:hyperbola}).  Suppose the hyperbola has semimajor axis $a>0$ and eccentricity $e> 1$.  Then  the  Chazy parameters $A, A'$
for this orbit  are vectors with equal length $\sqrt{2h}$ pointing along the asymptotes of the hyperbola as in Figure~\ref{fig:hyperbola}. 
($\E$ can be  parameterized by  $q_2 -q_2$.)  The vectors $C, C'$ are orthogonal with equal lengths proportional to $a\sqrt{e^2-1}$ (see Example~\ref{ex_Kepler} below).  Together, $A,C$ uniquely determine the hyperbola and then $(A',C')=F(A,C)$ are also uniquely determined.  The asymptotic shape vectors $A/\normtwo{A}$ and $A'/\normtwo{A'}$   encode the incoming and outgoing directions and the main problem reduces to:  For a fixed incoming direction, which outgoing directions are possible ?  A little thought about hyperbolas shows that the outgoing unit vector can be anything other than minus the incoming one.  To put it another way, all possible ``scattering angles'' are possible, except $\pi$.  Of course,  the situation for the $n$-body problem with $n>2$ is much more complicated.

\begin{figure}[h]
\scalebox{0.4}{\includegraphics{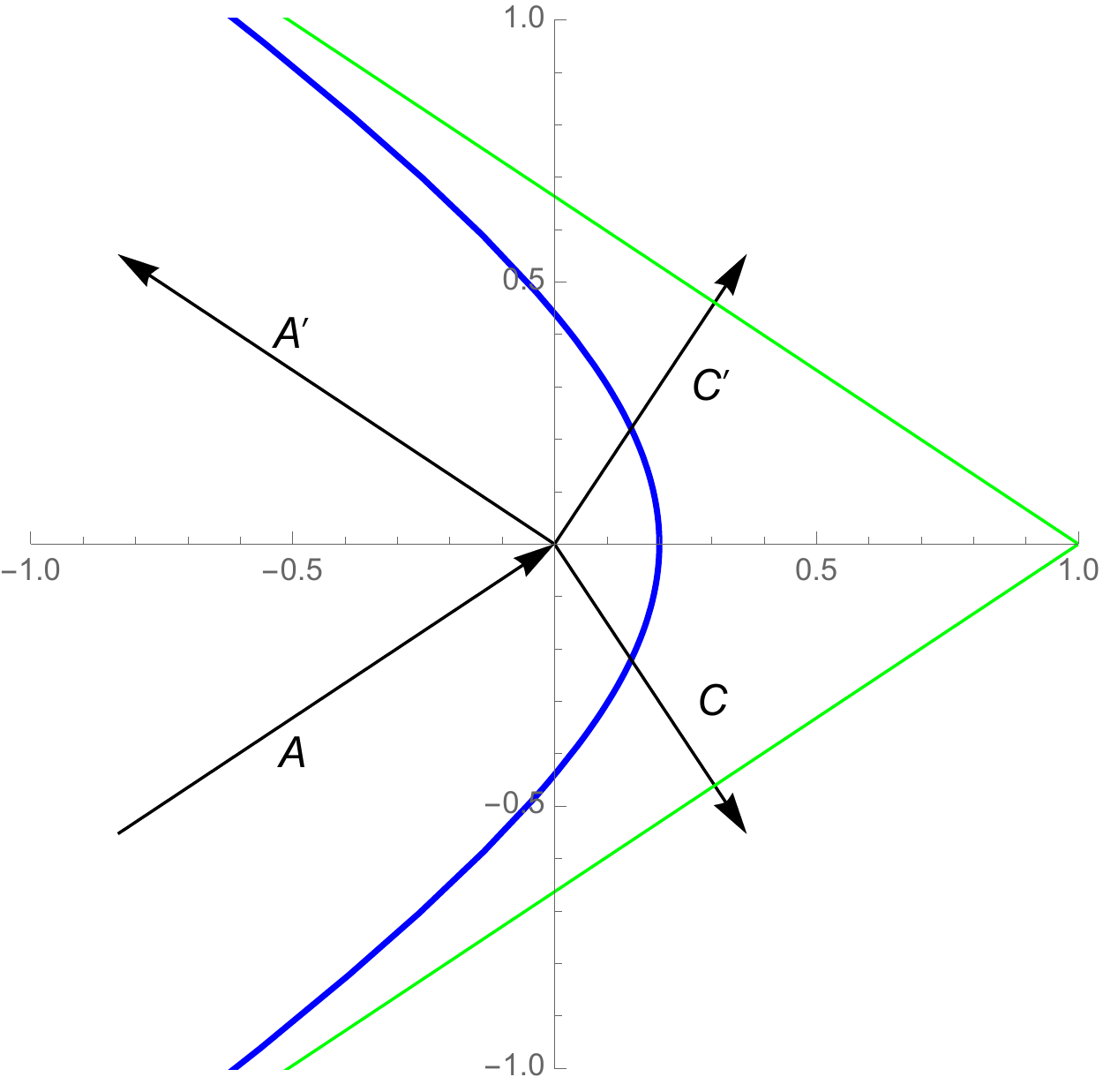}}
\caption{Hyperbolic Scattering for the Kepler Problem}
\label{fig:hyperbola} 
\end{figure}

We approach scattering by constructing a kind of McGehee regularization 
which partially compactifies $ n $-body phase space by adding a boundary at infinity.  Let $r$ be the square root of the moment of inertia of the bodies with respect to its center of mass, which represents the overall size of the configuration.  Then the boundary at infinity is given by $r=+\infty$ or equivalently, $\rho: = 1/r = 0$.
The regularized  dynamics extends analytically to this infinity manifold.  
The infinity manifold is an invariant manifold for the extended flow and contains a submanifold  $\hat{\CE}$ of 
equilibrium points  which gives the escaping scattering solutions a place to go.  (The hat over $\CE$   indicates
that we have removed collisions from a slightly larger  equilibrium set  $\CE$ at infinity.) 
$\CEh$ is normally hyperbolic, has half the dimension of the phase space and falls into two components, $\CEh_-$ and $\CEh_+$, one for the infinite past, the other for the infinite future, with $\CEh_-$ repelling and $\CEh_+$ attracting. Every future-hyperbolic solution has for its $\omega$-limit set a single point of $\CEh_+$.  
Every past-hyperbolic solution has for its $\alpha$-limit set a single point of $\CEh_-$.  These sets of restpoints  $\CEh_{\pm}$ can be  identified with the configuration space $\hat\E$ where the hat symbol means that we have deleted the collision configurations.  With this identification, specifying a restpoint $p\in\CEh_-$ is equivalent to specifying the Chazy parameter $A$ at time minus infinity. and similarly $q\in\CEh_+$ corresponds to a Chazy parameter $A'$ at time infinity.  From this point of view, the scattering parameters parameters $C,C'$ specify particular orbits in the stable and unstable manifolds of the restpoints determined by $A,A'$.
For the blown-up flow, the bi-hyperbolic solutions are exactly the heteroclinic orbits  connecting
$\CEh_-$ to $\CEh_+$.    If $p \in \CEh_-$ and $q \in \CEh_+$ are connected by a bi-hyperbolic orbit we will write $p \rightarrow q$  which is equivalent to saying $A\rightarrow A'$ for the corresponding Chazy parameters, as above.  

By the {\it image} of $A$ under the scattering relation we will mean the set  $\CHS(A) = \{A': A \rightarrow A'  \}$.  
Then the main problem of scattering can be stated as: ``what is the nature of this image set?''.   Besides preserving the energy levels, there seem to be no other obvious restrictions and indeed, for $n=2$, almost all energetically possible scatterings do occur.  Nevertheless, for $n\ge 3$,  it is challenging to prove anything at all.

Here is an outline of the paper. Section 2 sets up the equations of motion and introduces the blown-up coordinates at infinity.   In Section 3 we study the resulting flow on the invariant manifold at infinity.   At the end of section 3 we  formulate our key  analytic linearization result, Theorem~\ref{th_analyticlinearizationv1}, 
which is needed  in Section 4 to obtain a new proof of Chazy's asymptotic formulae.  We defer the proof of this analytic linearization result to 
the appendix where its proof  becomes a  corollary of the more general Theorem~\ref{th_analyticlinearization}. 
        
Section 5 forms the heart of our investigation.  There we define the scattering map and  the scattering relation and prove some of their basic properties.
Then  we focus on scattering near infinity, that is,  on the behavior of bi-hyperbolic solutions whose overall  size $r$ remains large throughout.   Our main result about the main problem shows that using heteroclinic orbits near infinity, we can connect almost every initial $A$ to  at least an open subset of the energetically possible $A'$:
\begin{theorem*}
For generic $A\in\hat \E$, the image $\CHS(A)$ has nonempty interior in the sphere $\{A':\normtwo{A'} = \normtwo{A}\}$.
\end{theorem*}  

{\sc Related Work.}  The literature on classical scattering is vast.  See the somewhat encyclopaedic \cite{Dz}.  However, there is very little rigorous work
that we are aware of on global properties of Newtonian N-body scattering for $N> 2$.  Chazy \cite{Chazy} is the seminal piece.
Quite recently Maderna and Venturelli  \cite{Maderna} posted  an elegant variationally  inspired article which is  closely related
but  complementary to this present work. 
Their main result is that given any asymptotic shape,
an ``$s = A/\|A\|$'' in Chazy language, and any configuration $q_0 \in \E$, there is a hyperbolic orbit passing through $q_0$ at
time $t =0$ and whose asymptotic shape is $s$ as $t \to +\infty$.
In other words, instead of trying to connect a given state at $t = - \infty$ to one at $t = + \infty$ they
 connecting finite states at $t=0$ to states at $t = + \infty$. 
  
{ \sc Acknowdledgements.}   
This work was started at MSRI during their semester on Hamiltonian Systems in the Fall of 2019 and all of us would like to acknowledge their support and the wonderful working environment at MSRI.   
We would like  to acknowledge useful discussions with other MSRI members attending a seminar on scattering at MSRI, notably Tere M. Seara, Amadeu Delshams, Jim Meiss, and Pau Martin.  
Montgomery would like to acknowledge useful  discussions with Andreas Knauf and
Maciej Zworski. 
Moeckel would like to acknowledge support from NSF grant DMS-1712656, NCTS in Hsinchu, Taiwan, Universit\'e Paris-Dauphine and IMCCE, Paris.
Guowei Yu thanks useful discussion with Xijun Hu and Yuwei Ou, and the support of Nankai Zhide Foundation. Nathan Duignan thanks the support of the Australian government's Endeavour Fellowship.

\section{Change of variables extending the flow to infinity} 

Let $m_i>0, q_i\in\R^d$ and $\xi_i\in\R^d$ denote the masses, positions and velocities of the bodies and let $ q=(q_i)_{i=1}^{n}\in \R^{nd }$ and $\xi =(\xi_i)_{i=1}^{n} \in\R^{nd}$.  The Newtonian potential function (the negative potential energy) is 
$$U(q)= \sum_{1 \le i < j \le n} \frac{m_i m_j}{|q_i-q_j|}$$
and the kinetic energy is 
$$K(\xi) = \frac12 \sum_{i=1}^n m_i|\xi_i|^2 = \frac12 \metric{\xi,M \xi}$$
where $M = \diag (m_1,\ldots,m_1,m_2,\ldots,m_2, \ldots, m_n,\ldots,m_n)$ is the $nd \times nd$ mass matrix with $d$ copies of each mass along the diagonal and $\metric{\cdot , \cdot }$ is the Euclidean metric on $\R^{nd}$.

Introduce the \emph{mass inner product}, also called the  \emph{kinetic energy inner product}:
\begin{equation}
\label{eq_metrictwo}  \metrictwo{v,w} = \metric{v,Mw}.
\end{equation}
If we write the  norm  of a vector $v$ relative to this inner product  by  $\normtwo v ^2$
then  $K(\xi) = \frac12\normtwo{\xi}^2$ and 
$\normtwo{q}^2$ is the moment of inertia with respect to the origin.

 Newton's equations can be written  
   $$\ddot q = M^{-1}\nabla_{\text{euc}} U(q) = \nabla U(q)$$
where $\nabla_{\text{euc}}$ is the Euclidean gradient and $\nabla = M^{-1} \nabla_{\text{euc}}$ is the gradient with respect to
the mass inner product.  
Newton's equations preserve the total energy
$$H(q,\xi) = K(\xi)-U(q) = h,$$
which we assume to be positive throughout the paper. 

Without loss of generality, we may fix  the center of mass at the origin, and insist that the total linear momentum $\Sigma_{i=1}^n m_i \xi_i$  is zero.   
Then $q, \xi \in \E$ where
\begin{equation}
\label{eq_E} \E=\{q\in\R^{nd}: \sum_{i=1}^n m_iq_i=0\}
\end{equation}
This $\E$ is a subspace of $\R^{nd}$ of dimension $D=d(n-1)$. Denote the collision set by  
$$\Delta = \{q: q_i=q_j\text{ for some }i\ne j\}.$$ 
Then the configuration space for the n-body problem is $\Eh = \E\setminus \Delta$ and its phase space is the tangent bundle $T\Eh =  \Eh \times \E$.

Set $S = \{q \in \E:  \|q \| =1 \}$ and introduce   spherical variables $ (r, s) \in (0 , +\infty]\times S$ on $\E$ according to: 
$$q = r s,  \text{ where } r = \|q \|,  \;\; s=q/r \in S ,$$
where $S$ is  a sphere of dimension $D-1$. Set $\Sh : = S \setminus \Delta$.  

Decompose  the velocity vector $\xi := \dot q$   as
$$\xi  = vs +w\qquad \text{ where } v= \metrictwo{s,\xi }\text{ and }\metrictwo{s,w}=0,$$
so that $v\in\R$ is the  radial velocity  component while $w\in\E$ is the tangential velocity.

Make the change of independent (time) variable
\begin{equation}
\label{eq_tau} d t = r d \tau,
\end{equation}
writing $' = \frac{d}{d \tau}$.
Then   Newton's equations become
\begin{equation}\label{eq_oder}
 \begin{aligned}
 r'&=v r\\
 s'&=w\\
 v' &= \normtwo{w}^2 - \frac1r U(s) \\
 w'&= \frac1r \nabla U(s)+\frac1r U(s)s -vw-\normtwo{w}^2s = \frac1r \tilde\nabla U(s)-vw-\normtwo{w}^2s
 \end{aligned}
 \end{equation}
 where 
 \begin{equation*}
 \label{eq_tilnabU}  \tilde\nabla U(s) = \nabla U(s) + U(s)s
 \end{equation*}
 is the tangential component of $\nabla U(s)$, since $U$  is homogeneous of degree $-1$.
 The energy equation becomes
 \begin{equation}
 \label{eq_energy}  \frac12 v^2+\frac12 \normtwo{w}^2 - \frac1r U(s) = h.
 \end{equation}
 The changes of variables and timescale leading to (\ref{eq_oder}) are a variation on the McGehee blow-up method \cite{McGeheeTriple, McGeheeICM} first introduced for studying triple collision.  In that case, the change of timescale was $d t = r^{\frac 32} d \tau$.  The factor of $r$ here is more appropriate for studying hyperbolic orbits near infinity.
 
As is well-known, any nonsingular solution with $h > 0$ must tend to infinity:
\begin{proposition}\label{prop_infinity}
Let $q(t)$ be a collision-free solution of the $ N $-body problem, with energy $h >0$,  defined for all $t\in [0,+\infty)$. Let $\tau=\tau(t)$ be the new time parameter defined as \eqref{eq_tau} with 
$$ \tau(0)=0, \;\; \tau_0= \lim_{t \to +\infty} \tau(t). $$
Then $r(t)\into+\infty$ as $t\into+\infty$ and $r(\tau)\into+\infty$ as $\tau\into\tau_0$.    Moreover, if $U(q(t))$ is bounded, then $\tau_0=+\infty$.  Similarly for solutions defined for all  $t\in(-\infty,0]$.
\end{proposition}

\begin{proof}
Let $I = r^2$ be the moment of inertia. Using the usual timescale we have
$$\frac 12 \ddot I = r\ddot r + \dot r^2=v'+v^2$$
where we used the fact that $\dot r = v$ and the definition of the new timescale.
Using the energy equation and the differential equation (\ref{eq_oder}) for $v'$   we derive
$$v' + v^2 =  2h +\frac1r U(s)=2h + U(q),$$
so that 
$$\frac12 \ddot I \ge 2h.$$
Integrating  twice with respect to $t$ gives
$$I(t) \ge 2h t^2 + kt +l$$
where $k = \dot I(0),l =I(0)$.  Since $h>0$ and $I(t)= r(t)^2$, we have $r(t)\into+\infty$ as $t\into+\infty$.  In fact
$$r(t) \ge ct$$
for $t$ sufficiently large, where $c>0$ is a suitable constant.

If we have an upper bound $U(q(t))\le K$ we get  upper estimates
$$\frac12 \ddot I \le 2h+K\qquad r(t) \le Ct$$
for $t$ sufficiently large and some $C>0$.
Then
$$\frac{d\tau}{dt} = \frac{1}{r(t)} \ge \frac{1}{Ct}$$
Hence $\tau(t)\into+\infty$ as $t\into+\infty$ and the rescaled solution exists for $\tau\in [0,+\infty)$.
 \end{proof}

For studying behavior near infinity, introduce
 $$ \rho = 1/ r.$$
which transforms  equations \eqref{eq_oder} into 
\begin{equation}\label{eq_odex}
 \begin{aligned}
 \rho'&=-v \rho\\
 s'&=w\\
 v' &= \normtwo{w}^2 -\rho U(s)\\
 w'&= \rho \tilde\nabla U(s)-vw-\normtwo{w}^2s
 \end{aligned}
\end{equation}
with energy equation
$$\frac12 v^2+\frac12 \normtwo{w}^2 -\rho U(s) = h.$$

Define  the infinity manifold to be 
$$\Sigma = \{(\rho,v,s,w): \rho = 0,\, v \in \R,\, s \in S,\, w \in \E,\, \metrictwo{s,w}=0,\, h>0 \}.$$
$\Sigma$ is an invariant manifold of \eqref{eq_odex} which represents the behavior of the $n$-body problem ``at infinity''.
Topologically,  $\Sigma$ is a manifold of dimension $2D-1$, diffeomorphic to $S \times (\R^D\setminus \{0\})$.  To see
its topology, use $\xi = w + v s \in \R^D$ as above, and note that when $\rho =0$ the positive energy is given by $h = \frac12 \normtwo{\xi}^2$.
Fixing this  energy to be the positive number  $h$  defines  a $(2D-2)$-dimensional submanifold $\Sigma_h \subset \Sigma$ which is the product of two
 $(D-1)$-spheres.   
 
The usual phase space $\hat \E \times \E$  is coordinatized by those $(\rho, s, v, w)$ for which $\rho > 0$
while $s \notin \Delta$ with the coordinatization being $(\rho, s, v,  w) \mapsto (q, \xi) = (\frac{1}{\rho} s, v s + w)$.
We will assiduously avoid collisions in our investigations here.    
To this end   set 
$$ \hat \Sigma :=\{ (0, s, v, w) \in \Sigma: \; s \in \hat S \}. $$

\begin{definition} We will refer to $(0, +\infty) \times \hat \Sigma$ as the ``usual'' or ``original'' phase space and  $[0, +\infty)\times \hat \Sigma $  as the ``blown-up'' or ``extended'' phase space. 
\end{definition}

\begin{remark}
	Extended phase space is a manifold with boundary,  this   boundary being  the infinity manifold $\hat{\Sigma}$. 
	Equations \eqref{eq_odex} make sense and are analytic for $\rho < 0$  as well, so we could    further extend  phase space to  $\R \times \hat{\Sigma}$. However $\rho < 0$ has no  physical meaning that we are aware of and this extension will rarely, if ever be  considered in what follows. 
\end{remark}

The idea of introducing an invariant manifold at infinity goes back to McGehee \cite{McGeheeStable, MoserStable, EastonMcGehee}.  It has been used many times since then, mainly in the negative or zero energy cases \cite{RobinsonHomoclinic, RobinsonSaari, MMS}, but also  in \cite{HuEtc} for the
positive energy case.

\section{Flow on and near the infinity manifold}
The differential equations on the infinity manifold $\Sigma$ arise by setting $\rho = 0$ in \eqref{eq_odex} and are 
\begin{equation}\label{eq_odeinfinity}
 \begin{aligned}
 \rho'&=0\\
 s'&=w\\
 v' &= \normtwo{w}^2 \\
 w'&= -vw-\normtwo{w}^2s
 \end{aligned}
\end{equation}
The energy equation with $\rho = 0$  becomes
\begin{equation}
\label{eq_energyinfty} v^2+ \normtwo{w}^2 = 2h.
\end{equation}
Notice that the potential $U(s)$ has completely disappeared from the equations.   $\Sigma$ contains submanifolds of equilibria $\CE = \CE_- \cup \CE_+$, where 
$$ \CE_{\pm} = \{(0, s, v, 0): \; s \in S,\, v = \pm \sqrt{2h},\, h>0 \}, $$
These equilibrium points are the   alpha (-)  and omega (+)  limit sets of hyperbolic  orbits.

{\sc  Free particles and the flow at infinity.} Since the flow at infinity is independent of the potential, we can understand it
by setting the potential  to zero, i.e., by considering  the motion of free particles.  
Then each body moves in a straight line at constant speed with respect to the usual timescale, so $q_i(t)= A_i t+ C_i$ where $A_i,C_i\in\R^d$ and
$m_1A_1+\ldots+ m_nA_n = m_1C_1+\ldots+  m_nC_n = 0$, 
or 
\begin{equation}\label{eq_freeparticle}
q(t) =  At+C
\end{equation}
where $A = \dot q \in \E$ is the constant velocity vector.  The vector $C$ is not unique, since a time translation  $t \mapsto t-t_0$ 
transforms  $C \mapsto C - t_0 A$.  Using such a translation  we may assume $\metrictwo{A,C}=0$.  The corresponding $C$ represents
the impact parameter: the closest point on the straight line to the origin. 
Now clearly (\ref{eq_odeinfinity}) is just a rescaled version of this free particle flow.

Consider the asymptotic behavior of a free particle solution \eqref{eq_freeparticle}
  as $t\into\pm\infty$ or equivalently $\tau\into\pm\infty$, but written out  in our ``McGehee'' coordinates.
We find  
$$(\rho(\tau),s(\tau),v(\tau),w(\tau))\into (0,\pm\frac{A}{\normtwo{A}}, \pm \normtwo{A},0)\qquad \tau\into\pm\infty.$$
Thus, in blown-up coordinates,  the free particle motion is  a heteroclinic orbit  connecting restpoint $p$ to restpoint $-p$ where 
$$p =  (0,-\frac{A}{\normtwo{A}},-\normtwo{A},0)\in\CE_-, \qquad -p =  (0,\frac{A}{\normtwo{A}},\normtwo{A},0)\in\CE_+.$$
When projected onto the sphere, $s(\tau)$ is  half of a great circle and connects the antipodal pair  $\pm \frac{A}{\normtwo{A}}$.
See Figure \ref{fig:infity}.

\begin{figure}[ht]
\scalebox{0.6}{\includegraphics{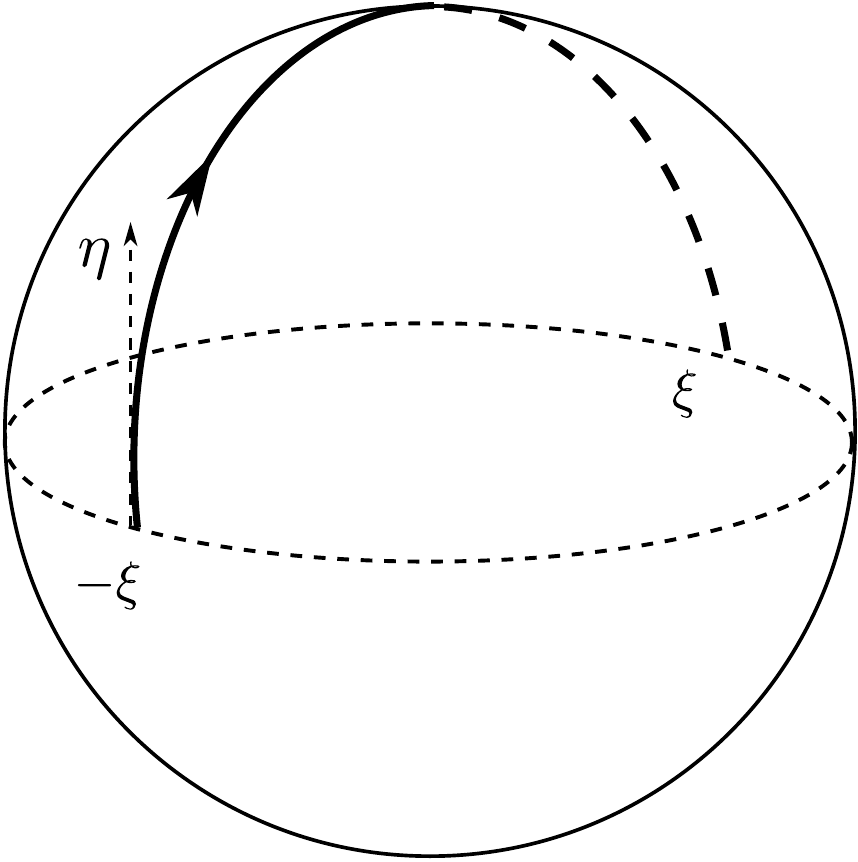}}
\caption{A half great circle solution}
\label{fig:infity} 
\end{figure}

The following proposition  reformulates  the facts just described.

\begin{proposition}\label{prop_infinityflow}
\label{prop: flow at infty} For any $\xi, \eta \in S$ satisfying $\metrictwo{\xi, \eta}=0$, and any $h > 0$,
equations \eqref{eq_odeinfinity} and \eqref{eq_energyinfty}  are solved by 
\begin{equation*}
\begin{aligned}
\rho(\tau) & = 0 \\
v(\tau) & = \sqrt{2h} \tanh(\sqrt{2h}\tau) \\
s(\tau) &= \xi \sin(\tht(\tau)) + \eta \cos (\tht(\tau)) \\
w(\tau) &= \sqrt{2h} \sech(\sqrt{2h}\tau) (\xi \cos(\tht(\tau)) -\eta\sin(\tht(\tau))),
\end{aligned}
\end{equation*}
where $\tht(\tau)= tan^{-1}(\sinh(\sqrt{2h}\tau)).$
This solution is    heteroclinic, connecting 
$$ p=(0, -\xi, -\sqrt{2h}, 0)  \in \CE_- \qquad \text{ to }   -p=(0, \xi, \sqrt{2h}, 0) \in \CE_+. $$
Up to  time shift $\tau \mapsto \tau + \tau_0$,  every solution to \eqref{eq_odeinfinity} and \eqref{eq_energyinfty} is of  this form.
\end{proposition}

\begin{proof} The proof is a straightforward computation.  Alternatively, convert to the  case of a free particle  just described.
\end{proof}  

\begin{remark}  Figure \ref{fig:infity} and the discussion preceding proposition \ref{prop_infinityflow} describes the time $\pi$ geodesic
flow for the sphere.  Melrose and Zworski \cite{Melrose} 
underlined  the central role played by this flow for  quantum scattering.
\end{remark}

We   obtained the limiting equations (\ref{eq_odeinfinity}) by  setting $\rho=0$ in  (\ref{eq_odex}). However,   in the neighborhood of   collision points the values of the expressions $\rho U(s)$ and $\rho \tilde\nabla U(s)$ are not uniformly small.  Some of the ``half-great circle'' solutions in Proposition~\ref{prop_infinityflow} will  pass through collision configurations.   In order to correctly extend   into the interior $\rho > 0$, where
the real solutions lie, we will need to   delete these collision solutions. When we impose this restriction     we  obtain a dynamical system on the open subset 
of the infinity manifold 
  which is analytic, 
extends analytically into the interior, and   satisfies the usual existence, uniqueness and smoothness properties.

Let $\hat{\CE}= \CEh_- \cup \CEh_+$ be the subsets of collision-free equilibrium
points in the infinity manifold, where $\hat{\CE}_{\pm} :=  \CE_{\pm}\setminus
\Delta$ and $\Delta$ in this context is the extension of collision points to
infinity, that is, points when $s_i=s_j, i\neq j$. For  $p = (0,s_0,v_0,0) \in \hat{\CE}_{\pm}$, consider the linearization of \eqref{eq_odex} at $ p $:
\begin{equation}\label{eq_odelinearized}
\m{ \rho_1 \\ s_1 \\  v_1  \\  w_1}'= \m{-v_0 &0&0&0 \\ 0&0&0&I \\ -U(s_0)&0&0&0 \\ \tilde\nabla U(s_0)&0&0&-v_0 I}\m{ \rho_1 \\ s_1 \\  v_1  \\  w_1}
\end{equation}
where $I$ is the $D\times D$ identity matrix acting on  $\E$  and $\rho_1, s_1, \ldots$ represent linearizations of the corresponding variables $\rho, s, v, w$. 
The   variation vector $(\rho_1 , s_1, v_1  ,  w_1) \in \R \times \E \times \R \times \E$ arising in \eqref{eq_odelinearized}  should be restricted to the $2D$-dimensional invariant subspace
\begin{equation}
\label{eq:tgtSpace}
\mathcal{T} (s_0) := \{  \metrictwo{s_0,s_1} = \metrictwo{s_0,w_1}=0\} \subset \R \times \E \times \R \times \E
\end{equation}
which is  the tangent space at $(0, s_0, v_0, w)$  to our phase space $P \subset  \R \times\E \times \R \times \E$. 
At this point is worth observing  that 
our basic    ODE  
	\eqref{eq_odex}  is    defined  and analytic for all  $s, w \in \E$ with  $s \notin \Delta$. Within this larger ``thickened'' phase
	space $\R \times \hat \E \times \R \times \E$  the  equilibrium manifold  
	  gains a dimension, becoming    $\tilde \CE = 0 \times \hat \E \times \R \times 0 \supset \CE$.  
	  Linearizing  the ``thickened''   flow at any point
	 of this larger equilibrium manifold yields  the  matrix \eqref{eq_odelinearized}.   This full  matrix has a   $(D+1)$-dimensional eigenspace with eigenvalue $\lambda = 0$ whose   eigenvectors form the tangent space to $\tilde \CE$,  
	  which is to say, are of  the form
\begin{equation}
\rho_1 = w_1 = 0  \label{eq_0Espace}
\end{equation}
Its  only other eigenvalue is $\lambda = -v_0$ which    has a $D$-dimensional  eigenspace with eigenvectors of  the form
\begin{equation}
  \rho_1 =  v_1 = 0, \qquad   w_1 =-v_0 s_1, 
\label{eq_nonzeroEspace}
  \end{equation} 
and a single   independent  generalized eigenvector
\begin{eqnarray*}
\label{eq_generalized_eigenvector}
(\rho_1 , s_1, v_1  ,  w_1) & =& (1,0,\frac{1}{v_0} U(s_0), -\frac{1}{v_0} \tilde\nabla U(s_0))\\
 & =& G(s_0,v_0). 
\end{eqnarray*}
Restricting  these eigenspaces to $\mathcal{T} (s_0)$ is achieved by simply imposing the constraints defining  $\mathcal{T} (s_0)$
  yielding $T_{(0,s_0, v_0, 0)} \CE = \{ (0, s_1, 0, w_1);  s_1, w_1 \perp s_0 \}$ as the eigenspace for $0$ and a $D$-dimensional generalized eigenspace
for $-v_0$ 
consisting of a $(D-1)$-dimensional eigenspace and the generalized eigenvector $G(s_0, v_0)$.

The general solution of the linearized differential equation at an equilibrium point $p_0 = (0,s_0,v_0,0)$ can be found from the matrix exponential.  If 
$L(p_0)$ denotes the matrix in (\ref{eq_odelinearized}) then we have
$$
\exp(\tau L(p_0)) = \m{u(\tau)&0&0&0\\ \tilde\nabla U(s_0)\left((1-u(\tau))/v_0^2-\tau u(\tau)/v_0\right)&I&0&I(1-u(\tau))/v_0\\
U(s_0)(u(\tau)-1)/v_0&0&1&0\\\tilde\nabla U(s_0)\tau u(\tau) &0&0&u(\tau)I}
$$
where $$u(\tau) = \exp(-v_0\tau).$$

Restricting  the linearized flow to the generalized eigenspace $N_{(s_0, v_0)}$  for  $\lambda = -v_0$ and then
adding the result to the corresponding equilibrium point yields the linear approximation to the full flow.    
  This  generalized eigenspace $N_{(s_0, v_0)}$ consists of all vectors of the form
\begin{equation}
\label{eigenspace}
i_{(s_0, v_0)} (s_1, \rho_1) := \rho_1 G(s_0,v_0) + (0,s_1,0,-v_0 s_1).
\end{equation}
Apply $\exp(\tau L (p_0))$ to this  vector and add the result  to the equilibrium point   $(0,s_0,v_0,0)$  we obtain  the representation
\begin{equation}\label{eq_stableunstablelinearizedflow}
\m{ \rho (\tau) \\ s (\tau) \\   v (\tau)  \\ w (\tau)} =\m{0 \\s_0 \\v_0 \\ 0} +  u(\tau) \m{ \rho_1   \\ s_1  - (\rho_1 /v_0) \tilde \nabla U(s_0) \tau    \\  (\rho_1 /v_0) U(s_0)  \\ -v_0 s_1  - (\rho_1 /v_0)  \tilde \nabla U(s_0)  + \rho_1 \tilde \nabla U(s_0) \tau  } 
\end{equation}
 for this linear flow.
 
 It will  help  to be  explicit about our linearizing change of variables  $(\rho, s, v, w) \leftrightarrow (\rho_1, s_0, v_0, s_1)$
 which  we are advocating for and to give it a name. We call the change of variables   $J$ and  define it by the equality   
  $$(\rho, s, v, w) = (0, s_0, v_0, 0) +  i_{(s_0, v_0)} (s_1,  \rho_1),$$
  where  $i_{(s_0,v_0)}$ is the linear inclusion  \eqref{eigenspace} above which parameterizes $N_{(s_0, v_0)}$.
Thus 
  \begin{equation}
J: (s_0, v_0, s_1, \rho_1) \mapsto (\rho, s, v, w) 
\label{eq:alignedCoords}
\end{equation}
has the  component expression 
$$(\rho, s, v, w) = \left( \rho_1, s_0 + s_1,  v_0 + \frac{1}{v_0} U(s_0) \rho_1, -\frac{1}{v_0} \tilde \nabla U( s_0, v_0) \rho_1 - v_0 s_1 \right)$$

In the   new $(s_0,v_0, s_1, \rho_1)$  coordinates, the  linearized  flow reads:
  \begin{equation} 
  \label{linflow}
  L_{\tau}:  (s_0, v_0, s_1,  \rho_1) \mapsto ( s_0, v_0,  \tilde s_1 , \tilde \rho_1 )
  \end{equation}
  where 
  \begin{eqnarray*}
  \tilde s_1   & =  & u  (\tau; v_0)   s_1   -  \tau u  (\tau;v_0) \alpha(s_0,v_0) \rho_1  \\
  \tilde \rho_1 &  =  &   \qquad  \quad   \quad  \quad   u (\tau;v_0)    \rho_1 
  \end{eqnarray*}
  with     $$ u(\tau;v_0)   = \exp(-v_0 \tau), \text{ and }  \alpha(s_0,v_0)  = \frac{1}{v_0} \tilde \nabla U( s_0, v_0) .$$
  
 The following analytic linearization theorem will be key to the sections that follow.   The theorem serves to transfer the 
 $(s_0,v_0, s_1, \rho_1)$ coordinates  to the extended phase space,  converting  the nonlinear flow to the linearized one just described.

\goodbreak

\begin{theorem}\label{th_analyticlinearizationv1}
There is an analytic diffeomorphism $\Phi$  from a neighborhood of  the equilibrium manifold  $\hat \CE$ (collisions deleted!)  in  extended phase space
to a neighbhorhood of the zero section $\{ s_1 = 0, \rho_1 = 0 \}$ of the new $(s_0, v_0, s_1, \rho_1)$-variable space which 
conjugates  the n-body  flow $\phi_{\tau}$ and
the linearized  flow   $L_{\tau}$  (eq (\ref{linflow})): 
$$\Phi (\phi_{\tau} (p)) = L_{\tau} (\Phi(p))$$
for as long as $\tau$ is small enough so that the curves  $\phi_{\tau'} (p), L_{\tau'}(p)$ lie in their corresponding
neighborhoods for $0 \le \tau' \le \tau$.   Moreover, $\Phi^{-1} (s_0,v_0, s_1, \rho_1) = J(s_0, v_0, s_1, \rho_1) + O(|s_1|^2 + \rho_1 ^2)$,
where $J$ is the fiber-linear change of coordinates defined above by equation (\ref{eq:alignedCoords}).
\end{theorem} 

\begin{proof}  The proof is found in the Appendix, and follows rather directly from   the more general Theorem \ref{th_analyticlinearization} there.
\end{proof} 

\begin{remark} Note in particular that $(s_0, v_0) \mapsto \Phi^{-1}(s_0,v_0, 0,0) = (0, s_0, v_0, 0)$ parametrizes the equilibrium manifold $\hat{\CE}$.
	If we fix  this  equilibrium value $(s_0,v_0)$ and vary  $(\rho_1, s_1)$ then 
	the map $(\rho_1, s_1) \mapsto \Phi^{-1} (s_0,v_0, s_1, \rho_1)$ parametrizes the local stable  or unstable  
	manifold attached to $(0, s_0, v_0, 0)$ .   Back in
	the $(\rho,s,v,w)$ variables, the flow will be   given by (\ref{eq_stableunstablelinearizedflow}) up to terms of order $O(\rho_1^2+\normtwo{s_1}^2)$.
\end{remark}  

\begin{remark}
         \label{rmk:Alert} 
	The alert reader may have noticed   something  fishy about our change of variables \eqref{eq:alignedCoords}. The variable $s =s_0 + s_1$ 
	is supposed to lie on the sphere $S$, but if  $s_0 \in S$ and  $s_1 \in T_{s_0} S = s_0 ^{\perp} $ then  $s_0 + s_1 \notin S$
	unless $s_1 = 0$.    To escape this trap,
	  relax the constraint that $\|s \| =1$ and work on the ``thickened phase space''  described immediately after \eqref{eq:tgtSpace}.
	  Theorem \ref{th_analyticlinearizationv1} holds in this larger context, and that is where we will prove it, 
	  restated and somewhat generalized as  theorem  \ref{th_analyticlinearization} and proved in the appendix.  As already noted just after
	  \eqref{eq:tgtSpace},
	  the  eigenvalue structures of the linearization at an equilibrium for our original phase space and for its thickened version are nearly identical,
	  both having as their only  eigenvalues  $0$ and $-v_0$ with $0$'s eigenspace being  the tangent space to the equilibrium manifold and
	  $v_0$ enjoying a  single nonzero  generalized eigenvector and a codimension genuine eigenspace. 
	  Theorem  \ref{th_analyticlinearizationv1} 
	  becomes a   corollary of the more general theorem   \ref{th_analyticlinearization} and an inspection of the proofs of
	  these  theorems yields the  solution to this  ``$s = s_0 + s_1$'' puzzle.   In  a nutshell, for $s_1$
	small the  points $s_0 + s_1$ of 
	the affine space $s_0 + T_{s_0} S$ lies within $O(|s_1|^2)$ of the sphere $S$.  The $O(|s_1|^2 + \rho_1 ^2)$ error
	of $\Phi^{-1} = J + O(|s_1|^2 + \rho_1 ^2)$ occurring in  the  end of the statement of theorem \ref{th_analyticlinearizationv1} 
	(compare with the $N_p$ of  theorem \ref{th_analyticlinearization}) includes a  projection
	of the affine space $s_0 + T_{s_0} S$ onto the sphere $S$.

\end{remark}

\section{Chazy-type Asymptotics}\label{sec_asymptotics}

In this section, we reinterpret our results about the flow near infinity using  the Newtonian timescale,  and so
rederive Chazy's asymptotic expansion with   error terms. {We focus on the case $t \to +\infty$ in most of this
  section. At the end we describe the various adjustments needed for the $t \to -\infty$ case.} 
The Chazy  expansion,  discussed  in the beginning of the paper,  is 
\begin{equation}
q(t)  =  A t + B \log (t) + C + f(t; A, C),  \text{ as } t \to +\infty  
\label{eq_ChazyAsymptotics} 
\end{equation}
where  $A, B, C \in \R^{nd}$ are constant vectorial parameters. 
Chazy claimed that the error term had the form $f(t;A, C) = Q(1/t,\log t/t; A, C)$
where $Q(u,v; A, C) $ is a two-variable  power series in $u$ and $v$  having   no constant term, absolutely convergent in a neighborhood 
of  $u = v= 0$ and whose coefficients depend analytically on A and C.  Note that in particular such an $f$ is 
$O(\log t/t)$.

The necessity of the $\log(t)$ term in (\ref{eq_ChazyAsymptotics})  is seen by
trying out  the ansatz   $At + C$ in Newton's equations.  The equations fail to hold at  $O(1/t^2)$ as $t \to +\infty$.
Assuming instead    $A t + B \log (t) + C$  and plugging into Newton's equations yields   $B = B(A)$ as per:
\begin{equation}
B = -\nabla U (A)
\label{eq_ChazyAsymptoticsB} 
\end{equation}
in order for the solution to hold up to $O(1/t^2)$. 

The parameters $A, C \in  \R^{nd}$  
are fixed ``scattering parameters'' and determine  the analytic function $f$. 
The parameter $A$ simultaneously  encodes the   asymptotic velocities and asymptotic shape since
$$ A = \lim_{t \to +\infty}  q(t)/t = \lim_{t \to +\infty} \dot q (t).$$
In keeping with our hyperbolicity assumption, we insist that $A$ be collision-free:  $A \in \hat{\E}$.  
The parameter $C$ is not uniquely determined, since a time shift $q(t) \mapsto q(t-t_0)$ alters $C$.  By shifting the origin of time, we can find a unique $C$ with $\metrictwo{A,C}=0$.
In analogy with the free particle case,  $C$ can be viewed as an ``impact parameter'' -- a   measure of how close the straight line $At + C$
misses total collision.  But for the orbits in the Newtonian problem, this interpretation of $C$ is less clear.
  
Chazy \cite[eq.~(22)--(27)]{Chazy} established his expansion \eqref{eq_ChazyAsymptotics}  and its analyticity for  hyperbolic solutions. See also \cite{Gingold}.   
We will  derive  his   result   from the local behavior of solutions near the stable/unstable manifold at infinity as investigated in the previous sections.  We summarize the result now. 

\begin{theorem} 
\label{thm_Chazy} [Case $t, \tau \to +\infty$] Let $q(\tau)$ be the position vector as a function of time $\tau$ for any solution lying in the stable manifold of the equilibrium point $(0, s_0, v_0,0) \in \CEh_+$ 
and having linearized data $\rho_1, s_1$ with $\metrictwo{s_0,s_1}=0$ as per equation \eqref{eq_stableunstablelinearizedflow}.  
Then, reexpressed in terms of Newtonian time $t$,  as $t \to +\infty$ the solution $q(t)$ 
satisfies Chazy's asymptotics \eqref{eq_ChazyAsymptotics}  with coefficients given by
\begin{equation}
\label{eq_Chazy}
\begin{split}
A &= v_0 s_0\\
B &= -\nabla U(v_0 s_0)= -\nabla U (A)\\
C &= \frac{s_1}{\rho_1} -\frac{\log(\rho_1 v_0)}{v_0^2}\tilde \nabla U(s_0).
\end{split}
\end{equation}
The error term is of the form $f(t) = Q(1/t,\log t/t)$
where $Q$ denotes a two-variable power series beginning with no constant term
depending analytically on the parameters $s_0, v_0, s_1, \rho_1$ and converging when the arguments are sufficiently small.
\end{theorem}

{\sc Sketch of Proof.} From the analytic linearization Theorem \ref{th_analyticlinearizationv1} we have that $\rho, s, v, w$
are analytic functions of $\tau$ depending parametrically in an analytic way on the limiting values $s_0, v_0 $
and linearizing parameters $\rho_1, s_1$.  In particular $r = 1/\rho$ and $s$ are analytic  functions of $\tau$ whose expansions we can work out. 
Multiplying  $r(\tau)$ and $s(\tau)$ yields the position vector $q(\tau) = r(\tau) s (\tau)$   as an analytic function whose expansion we can work out
and which depends parametrically   on   $s_0, v_0$
and  $\rho_1, s_1$. 
To finish off the derivation we require analytic expansion of  $\tau$ as a function of Newtonian time $t$.

{\sc Proof in Detail.} From the analytic linearization theorem and from  \eqref{eq_stableunstablelinearizedflow}  
we have that $\rho$, $s$ and $v$ are analytic functions of $\tau$,
depending analytically on the limiting equilibrium values $s_0,  v_0$,
and analytic linearization variables $s_1, \rho_1$,  and having first few terms,

\begin{equation}
\label{eq_rho}
\begin{split}
\rho(\tau) & = \rho_1 u (\tau)  +d_2 u(\tau)^2 \tau^2 +d_1 u(\tau) ^2 \tau+ \rho_2 u (\tau)^2 +  f_3(u,\tau u) \\
s(\tau) & = s_0 +  s_1 u(\tau) - \rho_1 \nbtl U(s_0) u(\tau) \tau/v_0 +f_2(u,\tau u) \\
v(\tau) & = v_0 +\rho_1 U(s_0) u(\tau)/v_0 + g_2(u,\tau u) \\
\end{split}
\end{equation}
where here and in what follows, expressions like $f_p, g_p, \ldots$ denote unspecified convergent power series which begin with terms of order $p$ and converge when the arguments are sufficiently small. In this section we always set 
$$ u(\tau) = u(\tau; v_0)= \exp(-v_0 \tau) .$$

The value of the coefficient $\rho_2$ will be needed later.  Plug the above $\rho(\tau)$ and $v(\tau)$ into the first equation of \eqref{eq_odex}, we get 
\begin{equation} 
\label{eq_rho2}  d_1 =d_2 =0,  \;\; \rho_2 = \left( \frac{\rho_1}{v_0} \right)^2 U(s_0). 
\end{equation}

Using these we can get asymptotic estimates for $q(\tau) = r(\tau)s(\tau)$.  First consider $r(\tau)$.
\begin{lemma}
The size variable satisfies
\begin{equation}\label{eq_r}
r(\tau)  = \frac{1}{\rho_1}e^{v_0 \tau} -\frac{U(s_0)}{v_0^2} + u^{-1} h_2(u,\tau u)
\end{equation}
\end{lemma}
\begin{proof}
Recall the differential equation $r'=v r$.  Using (\ref{eq_rho}) we have
$$\begin{aligned}
r(\tau) &= r(0)\exp(v_0\tau)\exp\left(\int_0^\tau \rho_1 U(s_0) u(\lambda)/v_0 \,d \lambda\right)\exp\left(\int_0^\tau g_2(u(\lambda),\lambda u(\lambda)) \,d \lambda\right)\\
&=  \frac{1}{\rho_1}\exp(v_0\tau)\exp\left(\int_{+\infty}^\tau \rho_1 U(s_0) u(\lambda)/v_0 \,d\lambda\right)\exp\left(\int_{+\infty}^\tau g_2(u(\lambda),\lambda u(\lambda)) \,d \lambda\right)
\end{aligned}
$$
where we used the fact that $\lim_{\tau\into +\infty}r(\tau)\exp(-v_0 \tau) = 1/\rho_1$ which follows from the first equation  of \eqref{eq_rho}.
The second exponential integral  factor evaluates to 
$$\exp\left(- \rho_1 U(s_0) u(\tau)/v_0^2 \right) = 1 -\rho_1 U(s_0) u(\tau)/v_0^2 + p_2(u)$$
and the third is of the form
$$\exp\left(G_2(u(\tau),\tau u(\tau)\right) = 1+ r_2(u,\tau u).$$
We have used the fact that the integral of a series of the form $f_p(u(\tau),\tau u(\tau))$ is an expression of the same form.
Multiplying the factors and using the fact that $\exp(v_0\tau) = u^{-1}(\tau)$ gives the required formula.
\end{proof}

Multiplying out $q(\tau) = r(\tau)s(\tau)$ we now find
\begin{equation}
\label{eq_q}
q (\tau) = \frac{s_0}{\rho_1} e^{v_0 \tau} - \frac{\nbtl U(s_0)}{v_0}\tau + \left( \frac{s_1}{\rho_1} - \frac{U(s_0)}{v_0^2} s_0 \right) + u^{-1}k_2(u(\tau),\tau u(\tau)). 
\end{equation}

We will rewrite this expansion for $q$ using  $t$ as a time parameter  in place of $\tau$. For this we need the following lemma.
\begin{lemma} \label{lm_tauexpant}  Along  any solution in the stable manifold
of the equilibrium point $(0, s_0, v_0, 0)$,  $v_0 >0$,  the Newtonian time  $t$ 
and our  new time scale   $\tau$ are related, for all sufficiently large $t, \tau$ by 
 an invertible  analytic  change of variables
of    the form 
\begin{equation}
\label{eq_texpan}   t =  \frac{1}{\rho_1 v_0}e^{v_0  \tau} -  \frac{U(s_0)}{v_0^2} \tau + c + u^{-1}p_2(u,\tau u),
\end{equation}
where $c$ is an arbitrary constant. Moreover its analytic inverse is characterized by
\begin{equation}
\label{eq_uexpan}
e^{v_0  \tau} = \rho_1 v_0 t + \frac{\rho_1}{v_0^2} U(s_0) \log(t) + \frac{\rho_1}{v_0^2}U(s_0) \log(\rho_1 v_0)- \rho_1 v_0 c+t P_2(1/t,\log t/t),
\end{equation}
\begin{equation}
\label{eq_tauexpan} \tau = \frac{1}{v_0} \log(t) + \frac{1}{v_0} \log(\rho_1 v_0) + P_1(1/t,\log t/t)
\end{equation}
\end{lemma}

\begin{proof}
Using \eqref{eq_r}, we get \eqref{eq_texpan} by integrating both sides of the identity  $dt = r d\tau$. For simplicity, let 
\begin{equation}
\label{eq_ab} a = (\rho_1 v_0)^{-1}, \;\; b = U(s_0)/v_0^2.
\end{equation}
so that  \eqref{eq_texpan} becomes
\begin{equation} \label{eq_t}
  t = a e^{v_0 \tau} - b\tau +  c + u^{-1}k_2(u,\tau u) = au^{-1}\left[ 1-b\tau u/a + cu/a + l_2(u,\tau u)  \right]
  \end{equation}  
  Taking logs gives
  \begin{equation}\label{eq_logt} 
 \log t = v_0 \tau +\log(a) -b\tau u/a + cu/a + m_2(u,\tau u)
  \end{equation}  
  Now we claim that any convergent power series of the from $f_p(u,\tau u)$ can be rewritten as a convergent power series $g_p(\frac1{t},\frac{\log t}{t})$.  To see this, we view $u$ and $w=\tau u$ as independent variables and similarly for $s= 1/t$ and $\sigma = (\log t)/t$.   The formulas above can be used to express $(s,\sigma)$ as power series in $(u,w)$
  $$s = \frac{u}a  +\ldots\qquad \sigma = v_0w+\frac{\log a}{a} u + \ldots.$$
  Since the matrix of linear terms is nondegenerate, the two-dimensional inverse function theorem shows that we can invert the series to get series for $(s,\sigma)$ is terms of $(u,w)$.
  
  Using this information,  we can replace all of the higher order terms in equations (\ref{eq_t}) and (\ref{eq_logt}) by series in $1/t, (\log t)/t$.  For example, replacing all but the first two terms in  (\ref{eq_logt}) and solving for $\tau$ gives
  $$\tau = \frac1{v_0}\log t -\frac{\log a}{v_0} + P_1(1/t,\log t/t)$$
  which agrees with (\ref{eq_tauexpan}).  Using this equation to replace $\tau$ in  (\ref{eq_t}) and replacing the term $u^{-1}k_2(u,\tau u)$ by $tp_2(1/t,(\log t)/t)$ we can solve for $e^{v_0 \tau}$ to find
  $$\exp(v_0 \tau) = a^{-1}t  + \frac{b}{av_0}\log t - \frac{\log a}{av_0} -  a^{-1}c +t P_2(1/t,\log t/t)$$
as claimed in (\ref{eq_uexpan}).  Here we use the fact that a term of the form $f_1(1/t, (\log t)/t)$ is a special case of a term $t f_2(1/t, (\log t)/t)$. 
\end{proof}

\noindent
Completion of the Proof of Theorem~\ref{thm_Chazy}. The error term in  \eqref{eq_q} can be replaced by $t f_1(1/t, (\log t)/t)$ as in the proof of the last lemma.  Then
plugging \eqref{eq_uexpan} and \eqref{eq_tauexpan} into \eqref{eq_q}, we get 
$$q(t) = v_0s_0 t -\frac{\nabla U(s_0)}{v^2_0} \log t+C+ tQ_2(1/t,\log t/t) . $$
with 
$$C=\frac{s_1}{\rho_1} -\frac{\log(\rho_1 v_0)}{v_0^2} \tilde \nabla U(s_0) +\left(\frac{U(s_0)}{v_0^2}(\log(\rho_1v_0)-1) - c v_0   \right)s_0.$$
Here we have used the fact that $\nbtl U(s_0) = U(s_0)s_0 + \nabla U(s_0)$.   By choosing the constant $c$, which amounts to shifting the origin of time, we can
eliminate  the last term from $C$.   The new value satisfies $\metrictwo{A,C}=0$.  This finishes our proof of the formulas for $A,B,C$, once the reader notices 
$$\nabla U(A) =  \nabla U(v_0 s_0) = \nabla U(s_0)/v_0^2. $$ 

It remains to consider the error term.  The series $tQ_2(1/t,\log t/t)$ differs from the required series $Q(1/t,\log t/t)$ only in that it allows the presence of terms of the form
$\frac{(\log t)^{k+1}}{t^k}$, $k\ge 1$.  We will now show that such terms do not occur.  Assume that $q(t) = At + B\log t+C+ t g(s,\sigma)$ where $g(s,\sigma)$ is a convergent power series in $s=1/t$ and $\sigma=\log t/t$.  Note that $q(t)$ can be factorized as $q(t) = t(A+ B\sigma+Cs+ g(s,\sigma))$.  Then by the homogeneity of the Newtonian potential,
$$\nabla U(q(t))= s^2\nabla U(A+ B\sigma+Cs+ g(s,\sigma)).$$
It follows that $\ddot q(t)$ must also admit a factor of $s^2$.  Using the fact that $\dot s = -s^2$ and $\dot\sigma = s^2-s\sigma$ we find
$$\ddot q(t) = s^2 h(s,\sigma) + s\sigma^2 g_{\sigma\sigma}(s,\sigma)$$
where we have gathered together several terms which have a factor of $s^2$.  It follows that $g_{\sigma\sigma}(s,\sigma)$ must be divisible by $s$, so the series $g(s,\sigma)$ cannot contain any terms involving monomials $s^0\sigma^k$, $k\ge 2$.  In other words every monomial in $Q_2(1/t,\log t/t)$ has at least one factor of $1/t$ so $tQ_2(1/t,\log t/t)$ can be written in the form $Q(1/t,\log t/t)$ as required.

{\sc Negative time asymptotics.} Deriving the negative time asymptotics is not as simple as replacing $t \to + \infty$ by $t \to -\infty$ in the basic expansion (\ref{eq_ChazyAsymptotics}).
Indeed $\log(t)$ for $t< 0$ is problematic, either being a purely imaginary number or undefined.  Instead we will use (\ref{eq_ChazyAsymptotics}) together with time reversal symmetry to derive the correct formula.

Let $q(t)$ be any solution which is hyperbolic as $t\into -\infty$.  Then $\tilde q(t) = q(-t)$ is a solution hyperbolic as $t\into +\infty$ and we can apply what we already proved to see that there is an asymptotic expansion
$$\tilde q(t) = \tilde At+\tilde B\log t + \tilde C +\tilde  f(1/t, \log t/t)\qquad t\into +\infty$$ 
$$\begin{aligned}
\tilde A &= \tilde v_0  \tilde s_0\\
\tilde B &= -\nabla U (\tilde v_0  \tilde s_0 )= -\nabla U(\tilde A)\\
\tilde C &= \frac{\tilde s_1}{\tilde \rho_1} -\frac{\log(\tilde \rho_1\tilde v_0)}{\tilde v_0^2}\tilde \nabla U(\tilde s_0).
\end{aligned}
$$
By definition $q(t)=\tilde q(-t)$ so we automatically get 
$$q(t) = \tilde A(-t)+ \tilde B\log (-t) + \tilde C +\tilde  f(1/(-t), \log(-t)/(-t))\qquad t\into-\infty.$$ 
which we will write in the form
\begin{equation}
\label{eq_ChazyAsymptoticsNegative}
q(t) =  At+ B\log |t|+ C +  f(1/t, \log |t|/t)\qquad t\into-\infty
\end{equation}
with $A= -\tilde A$, $B=\tilde B$ , $C=\tilde C$ and $f(1/t, \log |t|/t) = \tilde f(1/(-t), \log(-t)/(-t))$.

It remains to express the coefficients $A,B,C$ in terms of  $s_0,v_0,s_1,\rho_1$.  \begin{lemma}\label{lemma_timereversal}
The parameters $s_0,v_0,s_1,\rho_1$ of the solution $q(\tau)$ are related to the parameters $\tilde s_0,\tilde v_0,\tilde s_1,\tilde \rho_1$ of $\tilde  q(\tau)$ by
$$s_0=\tilde s_0\qquad v_0 = -\tilde v_0\qquad s_1=\tilde s_1 \qquad \rho_1 = \tilde\rho_1.$$
\end{lemma}
Using this we have
\begin{equation}
\label{eq_ChazyNegative}
\begin{aligned}
A &= v_0 s_0\\
B &= \nabla U ( v_0 s_0 ) = \nabla U(A)\\
C &= \frac{s_1}{\rho_1} -\frac{\log(\rho_1(- v_0))}{ v_0^2}\tilde \nabla U(s_0)=\frac{s_1}{\rho_1} -\frac{\log(\rho_1| v_0|)}{ v_0^2}\tilde \nabla U(s_0).
\end{aligned}
\end{equation}

As a corollary we have: 
\begin{proposition} 
\label{prop_A=sv}
The Chazy parameter $A$ as $t \to \pm \infty$ is related to the corresponding equilibrium equation $(s_0, v_0)$
as $\tau \to \pm \infty$ by $A =v_0s_0$ regardless of whether we are considering the asymptotics as  $t, \tau \to + \infty$ or $t, \tau \to - \infty$.
\end{proposition}

\begin{proof}[Proof of  Lemma~\ref{lemma_timereversal}]
We have
$$s_0 = \lim_{t\into-\infty}\frac{q(t)}{|q(t)|} = \tilde s_0\qquad v_0 =\lim_{t\into-\infty}\frac{q(t)\cdot \dot q(t)}{|q(t)|} =  -\tilde v_0$$
but $s_1,\rho_1$ require a bit more effort.  

First consider $\rho_1$. From formula \eqref{eq_stableunstablelinearizedflow}, we should have
$$\rho_1=\lim_{\tau\into -\infty}\exp(v_0\tau)\rho(\tau)\qquad \frac{1}{\rho_1} =\lim_{\tau\into -\infty}\exp(-v_0\tau)r(\tau)$$
Using the known  formula \eqref{eq_r} for $\tilde q$ and $s_0=\tilde s_0, v_0 = -\tilde v_0$ to see that
\begin{equation*}
\begin{split}
r(\tau) = \tilde r(-\tau) &  = \frac{1}{\tilde \rho_1}\exp( \tilde v_0(-\tau)) - \frac{U(\tilde s_0)}{\tilde v_0^2} + \ldots \\ 
& = \frac{1}{\tilde \rho_1}\exp( v_0\tau) - \frac{U(s_0)}{v_0^2} + \ldots.
\end{split}
\end{equation*}
Multiplying by $\exp(-v_0\tau)$ and taking the limit shows $1/\rho_1 = 1/{\tilde \rho_1}$.

Moving on to $s_1$, note that from formula \eqref{eq_stableunstablelinearizedflow}, we should have
$$
\begin{aligned}
\tilde s_1 = \lim_{\tau\into +\infty}\exp(\tilde v_0\tau)(\tilde s(\tau)-\tilde s_0+(\tilde \rho_1/\tilde v_0)\tilde\nabla U(\tilde s_0)\tau)\\
s_1 = \lim_{\tau\into -\infty}\exp(v_0\tau)(s(\tau)-s_0+(\rho_1/v_0)\tilde\nabla U(s_0)\tau).
\end{aligned}
$$
Then we have
$$
\begin{aligned}
s_1 &= \lim_{\tau\into -\infty}\exp(v_0\tau)(s(\tau)-s_0+(\rho_1/v_0)\tilde\nabla U(s_0)\tau)\\
&= \lim_{\tau\into +\infty}\exp(v_0(-\tau))(s(-\tau)-s_0+(\rho_1/v_0)\tilde\nabla U(s_0)(-\tau))\\
&= \lim_{\tau\into +\infty}\exp(\tilde v_0\tau)(\tilde s(\tau)-\tilde s_0+(\tilde \rho_1/\tilde v_0)\tilde\nabla U(\tilde s_0)\tau)\\
&= \tilde s_1.
\end{aligned}
$$
The second line is just changing the variable from $\tau$ to $-\tau$ and the limit direction.  The third line uses the definition $s(-\tau) = \tilde s(\tau)$ and what we already know about $s_0, v_0, \rho_1$.  Two minus signs in the first factor and  last term cancel out.
\end{proof}

The following proposition summarizes the relation between the parameters of the asymptotic expansions of a forward hyperbolic solution and its time reversal using the notation $T$ for the operation of time reversal.  It follows immediately from the formulas  $A = -\tilde A, C= \tilde C$ and Lemma~\ref{lemma_timereversal}.
\begin{proposition} 
\label{prop:timereversal} If a solution $q(t)$ has Chazy asymptotic parameters $(A,C)$ in the distant future, then
its  time reversed  solution $Tq (t) = q (-t)$ has Chazy asymptotic parameters $(-A, + C)$ for the distant past.
If $q$, written in blown-up variables,  tends to the  equilibrium   $(0, s_0, v_0, 0) \in \CEh_+$ as $\tau \to + \infty$,   then $Tq$, written in blown-up variables, tends to the   
equilibrium   $(0, s_0, -v_0, 0) \in \CEh_-$ as $\tau \to -\infty$.  
\end{proposition}

 \begin{example}\label{ex_Kepler}
	For the planar two-body problem, $d=n=2$, we can use $q=q_2-q_1\in\R^2\setminus 0$ to parametrize $\hat\E$ which has dimension $D=2$.  The mass norm is proportional to the ordinary Euclidean norm:
	$$\normtwo{q}= \sqrt{\mu} |q|\qquad \mu = \frac{m_1m_2}{m_1+m_2}.$$  
	The timescale change $dt=\normtwo{q}\,d\tau$ is almost the same as the one usually employed to solve the Kepler problem as a function of the eccentric anomaly, so we modify the standard formulas for the solution, for example, those in \cite{Pollard}.  For the hyperbolic case, we have  
	$$\begin{aligned}
	r(\tau) &= \normtwo{q(\tau)} =  a\sqrt{\mu}(e \cosh(\omega \tau)-1)\\  q(\tau)&=(ae-a\cosh(\w\tau), a\sqrt{e^2-1}\sinh(\w\tau)) \\
	t(\tau) &=\frac{a}{\w}(e\sinh(\w\tau)-\w\tau)
	\end{aligned}
	$$
	where $\w=\sqrt{2h}$, $a = \frac{m_1m_2}{2h}$ is the semimajor axis, and $e$ is the eccentricity.  For simplicity, we have rotated the hyperbola so that the perihelion occurs on the positive $x$-axis, as in Figure~\ref{fig:hyperbola}.
	
	It's interesting to compute the scattering parameters $A,C,A',C'$ for these solutions, as shown in the figure.  Taking the limits as $\tau\into\pm\infty$ of $s(\tau) = q(\tau)/r(\tau)$ gives the asymptotic shapes in backward and forward time
	$$s_0 = \frac{1}{e\sqrt{\mu}}(-1,-\sqrt{e^2-1})\qquad s_0' = \frac{1}{e\sqrt{\mu}}(-1,\sqrt{e^2-1}).$$
	These are unit vectors in the direction of  the asymptotes. The limiting values of $v$ are $\mp \w$ respectively and so 
	$$A = \frac{\w}{e\sqrt{\mu}}(1,\sqrt{e^2-1})\qquad A' = \frac{\w}{e\sqrt{\mu}}(-1,\sqrt{e^2-1}).$$
	
	All of the vectors $s$ have length $|s| = 1/\sqrt{\mu}$ and the shape potential $U(s)$ is constant.  Hence the tangential gradient vanishes and the formula for $C,C'$ simplify to
	$$C=\frac{s_1}{\rho_1}\qquad C'=\frac{s_1'}{\rho_1'}.$$
	We have
	$$\rho_1'=\lim_{\tau\into\infty}\exp{\w\tau}\rho(\tau) = \lim_{\tau\into\infty}\frac{\exp{\w\tau}}{a\sqrt{\mu}(e \cosh(\w \tau)-1)}=\frac{2}{ae\sqrt{\mu}}$$
	and a similar computation gives the same result for $\rho_1$.  Further computation with the formulas above gives
	$$s_1'=\lim_{\tau\into\infty}\exp{\w\tau}(s(\tau)-s_0) = \frac{2}{e^2\sqrt{\mu}}(e^2-1,\sqrt{e^2-1})$$
	and similarly
	$$s_1=\lim_{\tau\into-\infty}\exp{(-\w\tau)}(s(\tau)-s_0) = \frac{2}{e^2\sqrt{\mu}}(e^2-1,-\sqrt{e^2-1}).$$
	Hence
	$$C = \frac{a}{e}(e^2-1,-\sqrt{e^2-1})\qquad C' = \frac{a}{e}(e^2-1,\sqrt{e^2-1}).$$
	
	The explicit calculation for $q(\tau)$ immediately gives the asymptotic expansion as $\tau\into\infty$
	$$q(\tau) = \exp(\w\tau)(\frac{a}{2},\frac{a\sqrt{e^2-1}}{2}) + (ae,0) + \ldots.$$ 
	Using the fact that $\tilde\nabla(U(s))=0$, a little algebra shows that this agrees with our formula (\ref{eq_q}) evaluated at $(s_0',\sqrt{2h},s_1',\rho_1')$.  
\end{example}

\section{Hyperbolic Scattering}\label{sec_hyperbolicscattering}
In this section we study orbits which tend to infinity hyperbolically in both forward and backward time. Recall the limiting set of these orbits is the equilibrium manifold $\CEh = \CEh_+ \cup \CEh_-$, where 
$$ \CEh_{\pm} = \{(0, s, v, 0): \; s \in \hat S,\, v = \pm \sqrt{2h},\, h>0 \}.$$
For simplicity, a point $(0, s,v, 0) \in \CEh$ will be written as $(s, v)$ in the rest of this section. 

As noted previously, at the infinity manifold the differential equation (\ref{eq_odeinfinity}) is independent of the potential function $U(s)$ and all its solutions appear to be nonsingular.  But once we move away from it, we will have to insist that our solutions avoid the singular set $\Delta$ of the potential.  So in what follows we will exclude the singular set $\Delta$ from  the domain of the flow $\phi_\tau$ of the blown-up differential equations \eqref{eq_odex}.   
Since collisions are inevitable for the collinear $ n $-body problem, we also require that the space dimension $d$ is at least 2.  

Having done this, it follows from the analytic linearization Theorem \ref{th_analyticlinearizationv1}  that $\CEh_+$ is an attractor for the flow on the extended phase space $[0, +\infty) \times \hat \Sigma$.  Moreover, its  local stable manifold $W^s _{loc}(\CEh_+)$  is a nonempty open subset of the extended phase space, which is analytically foliated into the individual local stable manifolds $W^s _{loc}(q)$ of the equilibrium points $q \in \CEh_+$.  Similarly $\CEh_-$ is a repeller whose local unstable manifold $W^u_{loc} (\CEh_-)$ is a nonempty open subset of the extended phase space analytically foliated into local unstable manifolds $W^u _{loc}(p)$ of its points $p \in \CEh_-$.  

Let 
\[\begin{aligned}
\CH_+ &= W^s(\CEh_+)=\{z:\phi_\tau(z)\into \CEh_+ \text{ as }\tau\into+\infty\}\\
 \CH_- &= W^u(\CEh_-)=\{z:\phi_\tau(z)\into \CEh_- \text{ as }\tau\into -\infty\}
 \end{aligned}\]
denote the corresponding global stable and unstable manifolds.  
Note that a point $z$ belongs to the global stable or unstable manifold if and only if there is some finite time such that the orbit enters the corresponding  local stable or unstable manifolds, $W^s _{loc}(\CEh_+)$ or $W^u_{loc} (\CEh_-)$.  
It follows that $\CH_\pm$ are also nonempty open subsets of the extended phase space.  
It follows from the asymptotic estimates in Section~\ref{sec_asymptotics} that orbits with initial conditions $x\in \CH_+$ tend to infinity hyperbolically in forward time in the sense of Chazy. Similarly, if $x\in \CH_-$, its orbit tends to infinity hyperbolically in backward time.  
Finally, let
$$\CH = \CH_- \cap \CH_+.$$
Orbits with initial conditions $z \in \CH$ are bi-hyperbolic, that is, they tend to infinity hyperbolically in both time directions.
\begin{proposition} \label{prop_CHnonempty}
$\CH$ is a nonempty open subset of the extended phase space. Moreover 
\begin{equation}
\label{eq_NonemptyCH} \CH \cap \big( (0, +\infty) \times \hat{\Sigma} \big) \ne \emptyset.
\end{equation}

\end{proposition}
\begin{proof}
  Since $\CH_\pm$ are open sets of extended phase space, $\CH = \CH_- \cap
  \CH_+$ is open. To see that $\CH$ is non-empty, recall that the extended phase
  space includes the boundary manifold at infinity and
  Proposition~\ref{prop_infinityflow} shows that there are simple solutions in
  this boundary manifold connecting $\CEh_-$ and $\CEh_+$. Since the space
  dimension is at least two, the singular set $\Delta\subset S$ has codimension
  at least two. It follows that most of these connecting solutions from
  Proposition~\ref{prop_infinityflow} do not encounter $\Delta$. The initial
  conditions of these orbits are points of $\CH$, so $\CH\ne \emptyset$. As
  $\CH$ is open, these bi-hyperbolic points on the infinity manifold
  $\{0\}\times\hat{\Sigma}$ must each be contained in some open ball
  $\mathcal{B}_p\subset\CH$. By a dimension count, for $\mathcal{B}_p$ to be
  open it cannot be contained in the infinity manifold
  $\{0\}\times\hat{\Sigma}$. Hence \eqref{eq_NonemptyCH} holds.
\end{proof}

\begin{remark}
 \eqref{eq_NonemptyCH} shows that   there exist  bi-hyperbolic solutions of the $n$-body problem that are close to infinity
 but do not lie within the  infinity manifold.  In other words, there are real bi-hyperbolic solutions. 
\end{remark}

If $x\in \CH_+$, then there is a unique equilibrium point $q = (s_+,v_+) \in \hat{\CE}_+$, such that $\phi_\tau(x)\into q$ as $\tau\into+\infty$.  This defines an analytic mapping $\pi_+:\CH_+\into \CEh_+$.  Similarly there is an analytic mapping $\pi_-:\CH_-\into \CEh_-$ assigning to each $x\in \CH_-$ its limiting equilibrium point $p$ as $\tau\into-\infty$.  

\subsection{The Scattering Relation} We are interested in the scattering  problem:  which equilibrium points $p\in \CEh_-$ and $q\in \CEh_+$ can be connected by bi-hyperbolic orbits.
\begin{definition} We say that $p\in \CEh_-$ and $q\in \CEh_+$  are {\em related by hyperbolic scattering} if 
$$\pi_-(x) = p \text{ and }\pi_+(x) = q\text{ for some }x\in\CH.$$
We denote this relation by $p\rightarrow q$.
The {\em hyperbolic scattering relation} is the corresponding subset $\mathcal{HS}\subset \CEh_-\times\CEh_+$.
We will also say that $x$ connects $p$ to $q$. 
\end{definition}

For any relation,  we can define the image and preimage of particular points.  
For $p\in\CEh_-$ and $q\in\CEh_+$ we have
$$\CHS(p) = \{q\in\CEh_+: p\rightarrow q\}\qquad \CHS^{-1}(q) = \{p\in\CEh_-:p\rightarrow q\}.$$
So $\CHS(p)$ is the set of all final states which can be reached from the initial state $p$ via hyperbolic scattering. 

The scattering relation can be reformulated in terms of the Chazy parameter $A$. 
Recall that the map  $(s, v) \mapsto A = vs$ relates the Chazy parameter of a hyperbolic solution to the corresponding equilibrium  $(s, v)  \in \hat{\CE}$ it converges to, see Proposition \ref{prop_A=sv}. So far we have only defined Chazy parameter $A$ for hyperbolic solutions of the $n$-body problem (solutions that are not contained in the infinity manifold in the blow up variables). However we can obviously extend this relation to heteroclinic solutions contained in the infinity manifold.

This map, upon restriction either to $\CEh_+$ or to $\CEh_-$ is a diffeomorphism onto $\Eh$ with inverse $A \mapsto \pm (A/\|A \|,  \|A \|)$.  
Thus we obtain an  induced scattering  relation on $\Eh$.  We continue to use the symbol ``$\rightarrow$'', as it is induced by our original relation. 
To be more specific, we define  $A \rightarrow A'$ if and only if 
$$ p = (-A/\|A \| , -\|A \|) \rightarrow q= (A'/\|A' \|,  \|A '\| ). $$
If $\gamma$ is a bi-hyperbolic orbit connecting $p$ to $q$ we will  also say that $\gamma$ connects $A$ to $A'$.  
The scattering relation, expressed in Chazy parameters, defines a ``relation'' in the standard sense of set theory: 
which is to say, a subset of $\Eh \times \Eh$.  

We have the following basic theorem concerning this scattering relation on $\Eh$.
\begin{theorem}    
\label{thm:scatteringreln}Let $\rightarrow$ denote the scattering relation on $\Eh$.   
For all $A, A' \in \Eh$ we have 
\begin{enumerate}
\item[(i).] {\bf energy conservation:} $A \rightarrow A' \implies \|A \| = \|A'\|$.
\item[(ii).] {\bf reflexivity:}    $A \rightarrow A $.
\item[(iii).] {\bf T-symm:} $A \rightarrow A' \iff -A' \rightarrow -A$. 
\item[(iv).] {\bf dilation:}  For any $k >0$, $A \rightarrow A' \iff k A \rightarrow k A'$.
\item[(v).] {\bf rotation:} For the diagonal action of any $R \in O(d)$, $A \rightarrow A' \iff R A \rightarrow RA'$.
\item[(vi).] {\bf reversibility:} $A \rightarrow A' \iff A' \rightarrow A$
\item[(vii).] {\bf openness:}  For generic $A \in \Eh$ the  image  $\CHS(A) := \{A':  A \rightarrow A' \}$ of $A$ under the scattering relation  contains  an open subset of the  sphere 
$\Sh (\| A \|)  = \{ A' \in \Eh : \|A ' \| = \|A \| \}$.  The closure of this open set contains $A$. 
\end{enumerate}
\end{theorem}

Of all these properties the openness property (vii) is the most nontrivial one. Proving it will require some serious work. 
The genericity requirement is as follows.  Suppose $A = vs$ for a restpoint $(s,v)$.  If $s$ is a planar configuration, that is, all bodies lie in one plane, then the condition is simply that $s$ not be collinear.  If $s$ is not planar, we are only able to say that $s$ must not lie in the zero set of a certain determinant.
We postpone the  proof of openness to Theorems \ref{thm:openimage} and \ref{thm:generic}.  

\begin{proof} [Proof of Theorem \ref{thm:scatteringreln}]    
(i).  If $A \rightarrow A'$ then there is a bi-hyperbolic orbit $\gamma$ that goes from $(-A/ \|A \|, - \|A \|)$ to $(A'/\|A' \|, \|A '\| )$.
Let $h$ be the energy of $\gamma$. By continuity, the energy of the $\al$ and $\omega$-limit of $\gamma$ in $\CE$ must be $h$ as well. By \eqref{eq_energyinfty}, this means $\|A \|^2 = \|A' \|^2 =2h$.

(ii). Set $p = (-A/ \|A \|, - \|A \|) \in \hat{\CE}_-$, by Proposition \ref{prop_infinityflow} and the proof given in Proposition \ref{prop_CHnonempty}, there is an orbit $\gamma \subset \hat \Sigma$ entirely contained in the infinity manifold, which connects $p$ to its antipodal equilibrium $-p = ( A /\|A \|, \|A \|) \in \CE_+$, so that  $p \rightarrow -p$. (See also the remark on free particles and the flow at infinity immediately preceding Proposition \ref{prop_infinityflow}).  Both $p$ and $-p$ have the same Chazy parameter $A$.  So $A \rightarrow A$.

(iii).  Suppose $A \rightarrow A'$ with a connecting bi-hyperbolic orbit $\gamma$. The time reversed path  $ T \gamma (\tau) = \gamma(-\tau)$ of $\gamma(\tau)$ is again a bi-hyperbolic orbit.  We saw in proposition \ref{prop:timereversal} that
$T \gamma$ has  past Chazy parameter   $-A'$, and an identical argument shows that it has   future Chazy parameter  $-A$. 
Thus $-A' \rightarrow - A$.  

(iv). The $ n $-body problem admits the dilational symmetries $\delta_{\lambda}$, $\lambda > 0$, which in Newtonian time is given by 
$$ q(t) \mapsto  ( \delta_{\lambda} q) (t) : = \lambda q (\lambda^{-3/2} t), $$
and in the rescaled time by 
$$ \delta_{\lambda} \gamma (\tau) = \lambda \gamma (\lambda^{-1/2} \tau). $$
The dilational symmetry can be worked out easily in our blown up variables where it acts on the equilibrium set $\CEh$ by $\delta_{\lambda} (s, v) = (s, \lambda^{-1/2} v)$
and hence it acts on Chazy parameters by $A \mapsto  \lambda^{-1/2} A$.
Setting $k = \lambda^{-1/2}$ and suppose that $A \rightarrow A'$ via a bi-hyperbolic orbit $\gamma(\tau)$. 
Then the path $\delta_{\lambda} \gamma(\tau)$ is also bi-hyperbolic and connects $kA$ to $kA'$ and so $k A \rightarrow kA'$.
The inverse relation is obtained using the inverse transformation $\delta_{\lambda} ^{-1} = \delta_{1/\lambda}$.

(v).  The proof here is almost identical to (iv).  Use the fact that the action of $O(d)$ maps solutions to solutions and
preserves their hyperbolicity.  So if $\gamma$ connects $A$ to $A'$ then $R\gamma$ connects $RA$ to $RA'$.

(vi). Write $P = - I_d$ (P for `parity') for  the action of  $- I_d $, the negative of the identity.   $- I_d \in O(d)$ and acts on configuration space by $q \mapsto -q$. By (v), we have $A \rightarrow A' \iff -A \rightarrow -A'$.  Now apply the  time reversal map T of   (iii) to get $-A \rightarrow -A' \iff A' \rightarrow A$. 

(vii).  See Theorem \ref{thm:openimage} and \ref{thm:generic}.  

\end{proof}

\subsection{The Scattering Map}

To say that the scattering relation $p\rightarrow q$ holds is to assert the existence of a connecting orbit lying in the intersection of the unstable manifold of $p$ with the stable manifold of $q$.  By specifying which unstable orbit connects with which stable orbit, we arrive at a refinement of the scattering relation called the {\em scattering map}, which we now describe.

Recall that the local stable and unstable manifolds $W^s_{loc}(\hat{\CE}_+)$ and $W^u_{loc}(\hat{\CE}_-)$ are open subsets of the extended phase space which can be parametrized by quadruples $(s_0,v_0,\rho_1,s_1)$ where the first two entries parametrize the equilibrium manifold $\hat{\CE}_+$ or $\hat{\CE}_-$ and the second two entries parametrize the stable and unstable manifolds of the corresponding equilibrium points.  
In these coordinates, the flow on the local stable and unstable manifolds keeps the first two entries fixed and is linear on the second two entries as in (\ref{eq_stableunstablelinearizedflow}). 
We will use a separate coordinate system of this type for each of  $W^s_{loc}(\hat{\CE}_+)$ and $W^u_{loc}(\hat{\CE}_-)$.  

To parametrize the {\em orbits} in $W^s_{loc}(\hat{\CE}_+)$ and $W^u_{loc}(\hat{\CE}_-)$, we should identify quadruples which map  to one another under the linearized flows.  The scattering parameter $C$ given in \eqref{eq_Chazy} and \eqref{eq_ChazyNegative} is an orbit invariant.   
Unfortunately it is not well-defined at $\rho_1=0$ so we will have to use a more complicated approach.  Instead of the parameters $(\rho_1,s_1)$, introduce the parameter 
$$\gamma = (\rho_1, s_1-\frac{\rho_1\log(\rho_1|v_0|)}{v_0^2}\tilde\nabla U(s_0)).$$

Note that 
$$
\gamma =  (\rho_1, \rho_1 C), \text{ if } \rho_1 \ne 0.
$$
Then one can check that the linearized flow identifies pairs $\g, \g'$ if and only if $\g'= k \g$ for some constant $k>0$.  Introducing the notation $[\g]$ for the corresponding equivalence classes, we have parametrizations of the orbits in  $W^s_{loc}(\hat{\CE}_+), W^u_{loc}(\hat{\CE}_-)$ by triples  $(s_0,v_0,[\g])$.  
We call these triples ``orbit parameters''.  
{ Note that an orbit is constant (an equilibrium point) if and only if $\gamma = 0$ and that an orbit lies in the infinity manifold $\Sigma$ if and only if $\rho_1 = 0, \gamma = (0,s_1)$ for some $s_1 \perp s_0$.} 
The orbit parameters are related to the Chazy scattering parameters $A,C$  by $A=v_0 s_0$ while $C$ is the ratio of the two components of $\g$ provided that the first coordinate $\rho_1$ of $\g$ is positive.    
The inverse relation is $(s_0, v_0, [\g]) = (\pm A/\|A\|, \pm \|A \|, [(1, C)])$ with the $\pm$ depending on whether we are parameterizing $W^s_{loc}(\hat{\CE}_+)$ or $W^u_{loc}(\hat{\CE}_-)$.  

Now we can introduce a map describing  the scattering behavior of bi-hyperbolic orbits.
\begin{definition} 
\label{def:scattering}
Let $(s_0,v_0,[\g])$ and $(s_0',v_0',[\g'])$ parametrize the space of non-constant orbits of  $W^u_{loc}(\hat{\CE}_-)$ and $W^s_{loc}(\hat{\CE}_+)$, respectively,
so that $\gamma, \g' \ne 0$.  Denote these orbit spaces by  $W^u_{loc}(\hat{\CE}_-)^{orb}$ and  $W^s_{loc}(\hat{\CE}_+)^{orb}$ 
respectively, and the open subspaces of these orbit spaces  consisting of those  orbits which also lie in $\CH$ 
by $(W^u_{loc}(\hat{\CE}_-)\cap \CH)^{orb}$  and $(W^s_{loc}(\hat{\CE}_+)\cap \CH)^{orb}$.
Then the  {\em hyperbolic scattering map} is the map
$$F: (W^u_{loc}(\hat{\CE}_-)\cap \CH)^{orb} \into (W^s_{loc}(\hat{\CE}_+)\cap \CH)^{orb}$$
$$F(s_0,v_0,[\g]) = (s_0',v_0',[\g'])$$
which sends  the local orbit parameters of a  bi-hyperbolic orbit $\gamma$  as $\tau\into-\infty$ to those of the same orbit as $\tau\into +\infty$.  
Alternatively,  if $\rho_1 > 0$ for $\g$ (and hence $\rho_1' > 0$ for $\g'$) we can write the map using Chazy parameters:   $$F(A,C) = (A', C'),  \qquad  \rho_1 > 0.$$
\end{definition}

It's clear that the scattering map is a refinement of the scattering relation.   For  $p = (s_0,v_0)\in \hat{\CE}_-$ and $q = (s'_0,v'_0)\in \hat{\CE}_+$, we have that $p\into q$ if and only if there exist orbit parameters $P= (s_0,v_0,[\g])$ and $Q=(s_0',v_0',[\g'])$ with $Q=F(P)$. 
In particular 
$$\CHS(p) = \{ q: (q,[\gamma']) = F(p, [\gamma]) \text{ for some } [\gamma], [\gamma'] \}.$$ 

Note also that  it follows from Theorem \ref{thm:scatteringreln}, property (i), that if   $F(P) = Q$  with $P$ and $Q$
 as just described, then   $v'_0 = - v_0>0$.
 
 {\sc  Topological Considerations regarding the  domain of $F$.}

The equivalence classes $[\g] = [\rho_1, w]$  which we have used to parametrize orbits are rays through the origin  in 
 the real  vector space $\R \oplus T_{s_0} S \cong \E$.
The space of all such rays in a real vector space is a sphere, so a copy of $S$ in our case.  
In this way we see that $W^u_{loc}(\CEh_-)^{orb} \cong \CEh_- \times S$  and $W^s_{loc}(\CEh_+)^{orb} \cong \CEh_+\times S$
where $W^u_{loc}(\CEh_-)^{orb}$, $W^s_{loc}(\CEh_+)^{orb}$ are as per definition \ref{def:scattering}. 

We are not really    interested in these entire local orbit spaces, but rather   only those rays for which  
$\rho_1 \ge 0$, since if $\rho_1 < 0$ the corresponding trajectory immediately enters into the non-physical
region $\rho < 0$ of our phase space.    Insisting $\rho_1 \ge 0$ determines a closed hemisphere within
$S$.    We have seen that the  equator $\rho_1 = 0$, the boundary of this hemisphere, represents trajectories tangent to the infinity manifold $\Sigma$
and that $F$ is the identity when restricted to this equator. The Chazy parameter $C \in T_{s_0} S = s_0 ^{\perp} \subset \E$   is an affine parameterization
of the open upper hemisphere $\rho_1 > 0$ via $C \mapsto [1,C]$.   See figure \ref{fig:halfsphere}.  The closed hemisphere itself is topologically a unit disc in $T_{s_0} S$,
and represents the ``radial'' compactification of the Chazy parameters, i.e. of  the tangent space $T_{s_0} S$.    We will  denote this disc by $D(T_{s_0} S)$
and write it as
$$D (T_{s_0}  S) = T_{s_0} S \cup \partial_{+\infty} D(T_{s_0} S)$$
with  $\partial_{+\infty} D(T_{s_0} S)$ corresponding to the equator, which is to say the directions tangent to $\Sigma$, and diffeomorphic to a sphere $S^{D-1}$.
We  denote  the corresponding  disc bundle over the sphere by $D(S) = \bigcup_s D(T_s S)$.  It is a real  analytic fiber bundle over $S$.
More generally,  
take any smooth manifold $M$ , form the vector bundle $\R \oplus TM$  over $M$ and then form the corresponding  disc bundle $D(M) \to M$ whose fibers
are the non-negative rays $\rho_1 \ge 0$ within the vector spaces  $\R \oplus T_m M$. 
  In our case we are interested in the disc bundle  $D(\hat S) \subset D(S)$.  The
boundary of our  disc bundle $\partial D (\hat S)$ is identified with trajectories at  infinity,  $\rho_1 = 0$. 

\begin{figure}[ht]
\scalebox{0.4}{\includegraphics{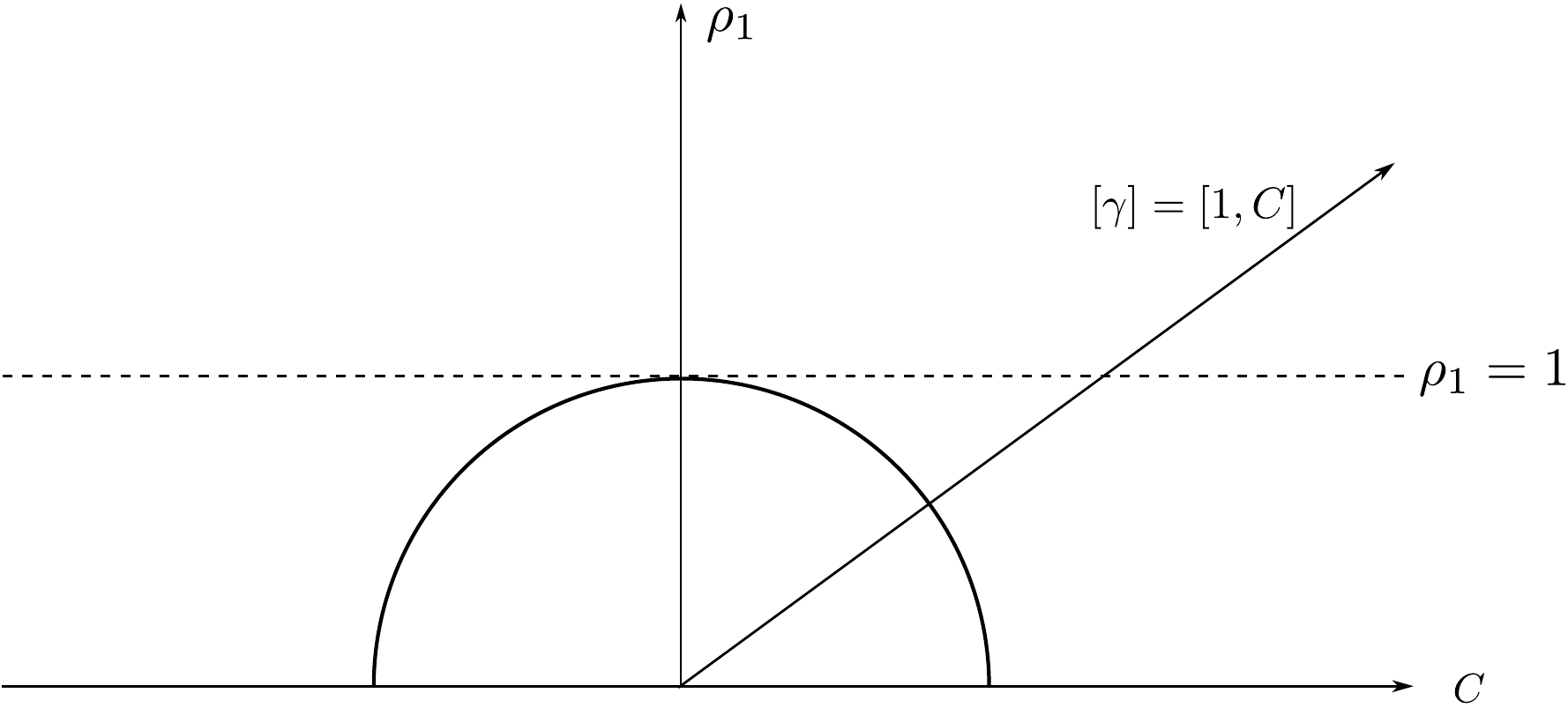}}
\caption{Affine coordinates on the rays through a half-space yield
a realization of the scattering parameter C.}
\label{fig:halfsphere} 
\end{figure}

{\sc Energy considerations.  }  We know the entire scattering map if we know its restriction to states
of energy $1/2$, which is the same as $\|A \| =1$, or $|v_0| = 1$.
Indeed we have seen that energy conservation implies that   $F(A, C) = (A', C') \implies \|A \| = \|A' \|$.
One verifies that the space time dilation, $\delta_{\lambda} q (t) = \lambda q (\lambda ^{-3/2} t)$ in terms of
Newtonian solutions, acts on Chazy parameters according
to $\delta_{\lambda} (A, C) = (\lambda ^{-1/2} A,  \lambda C)$.
It follows that  
$F(\delta_{\lambda} (A, C))=  \delta_{\lambda} (F(A,C))$ when $\rho_1 >0$.  

\begin{theorem}   The scattering map $F$ is determined by its restriction to those states
having energy $1/2$.  So restricted, the domain of the scattering map $F$ is
an  open subset of the disc bundle $D(\hat S) \to \hat S$  which was  obtained by compactifying the   tangent bundle $T\hat S$ as described above. 
 $F$ is analytic on its domain and equals the identity on the boundary of the disc bundle, $\partial D (\hat S)$
 which corresponds to   scattering along the infinity manifold. 
 Moreover, if $T(A,C) = ( - A,C)$ is the  time reversal operation as it applies to
  Chazy parameters (see Proposition \ref{prop:timereversal})  then   $T \circ F \circ T \circ F = Id$.  
\end{theorem}

 
\begin{remark}
	We  expect that   the  domain of $F$  is not  all of $D(\hat S)$:   orbits hyperbolic in the past need not be hyperbolic in the future.
	   Following \cite{Chazy}, call  an orbit
	``hyperbolic elliptic'' if one or more of the particle pairs $\{i, j \} \subset \{1, \ldots , n\}$ are bound in the distant future,
	meaning the   lim sup of their mutual distance $r_{ij}$ is bounded as $t \to + \infty$.     According to Alekseev (\cite{Alekseev} or Table 1 in p. 61 of \cite{ArnoldEncyc}) the set of 3-body orbits which are  hyperbolic 
	in the distant past and hyperbolic-elliptic in the distant future has positive measure.  Such  orbits almost certainly sweep out an open subset of phase space.  Assuming this openness, we can  picture   $D(\hat S)$ as being comprised  of two open sets separated
	by a closed measure zero boundary.  The two open sets  correspond to the  hyperbolic-hyperbolic type orbits  which form the domain of $F$, and the 
	 hyperbolic-elliptic type orbits.  They are separated by a  measure zero set  which contains   orbits ending in collision, 
	and, more notably perhaps, orbits  of   hyperbolic-parabolic type:  those for which the  bound pair motion  degenerates to a zero energy Kepler solution, so that $r_{ij} \sim t^{2/3}$.
	Understanding  the two open sets and their boundary   could be very interesting, but might be very difficult.
\end{remark} 
\color{purple}

\color{black}
 
  \begin{proof}   Analyticity of $F$ follows from the analytic linearization lemma, and the fact
  that the flow is analytic.   Except for the final assertion regarding $FTFT$, the rest of the theorem 
  was  proved in the discussion immediately preceding the theorem's statement. 
  
  To verify  $FTFT = Id$ use  the $(s,v,[\gamma])$ representation.  Suppose that a bi-hyperbolic orbit $\gamma \subset \CH$
  has  representation $(s_0, -v_0, [\gamma])$ in the infinite past and  $(s_0' , v_0 , [\gamma'])$
  in the infinite future so that $F(s_0, -v_0, [\gamma]) = (s_0', v_0 , [\gamma'])$. 
   The time reversal  $T \gamma$ of $\gamma$ is represented by  $(s_0 ' , - v_0 , [ \gamma'])$ in the infinite past,
  and $(s_0, v_0, [\gamma])$ in the infinite future, so that $F(s_0' , -v_0, [\gamma']) = (s_0, v_0, [\gamma])$.
  (Regarding the action of $T$, see Proposition \ref{prop:timereversal} and the paragraph preceding it.)
Now  $T(s, \pm v, [\gamma]) = (s, \mp v, [\gamma])$,   
  so   we  have just shown that $F(T(F(s_0, -v_0, [\gamma]))) = (s_0, v_0, [\gamma]) = T(s_0, -v_0, [\gamma])$, 
  or   $FTF = T$.
  Use $T^2 = Id$ to finish off.  
  
  \end{proof}

\subsection{Scattering at Infinity}
Understanding the full hyperbolic scattering map or relation is beyond our grasp right now.  It will be useful to restrict ourselves  to various subsets. Let $Z$ be any flow-invariant subset of   phase space.  Consider the smaller scattering map $F_Z = F|_Z$
and   relation
$$\CHS_Z= \{(p,q)\in \CE_-\times\CE_+: \pi_-(x) = p \text{ and }\pi_+(x) = q\text{ for some }x\in\CH\cap Z\}.$$
Clearly $\CHS_Z\subset  \CHS$.  We also use  the notations $\CHS_Z(p), \CHS_Z^{-1}(q), p\rightarrow_Z q$ with the obvious meanings.

Most of  what we know about  the scattering map arises from our precise understanding of the dynamics along the  infinity manifold $\Sigma$ as summarized by Proposition~\ref{prop_infinityflow}.  $\Sigma$ is  flow-invariant so we can form the restricted scattering relation $\CHS_{\Sigma}$ and map $F|_\Sigma$.
\begin{proposition} 
\label{thm:atinfinity}  The scattering map induced by the flow at infinity is obtained by restricting
$F$ to orbit parameters for which $\rho_1 = 0$, that is, $\gamma= (0,s_1)$.   In this case 
$$F_\Sigma(s_0,v_0,[0,s_1]) = (-s_0, -v_0,[0, s_1]). $$  
Expressed in partial Chazy parameters, this restricted $F$ is  the identity map
$$F(A, [0, s_1]) = (A, [0, s_1]).$$ 
The corresponding scattering relation is $\CHS_{\Sigma}(p) = -p$. 
\end{proposition}
\begin{proof}
The orbits at infinity are  the half-circles of Proposition~\ref{prop_infinityflow}.  We must  delete the ``singular orbits'', i.e., those which hit the collision locus.
 These orbits connect equilibrium points $p=(s_0,-v_0)$ to $-p=(-s_0,+v_0)$ having  the same value of the Chazy parameter $A$.  Since the ambient dimension $d$ is at least 2, not all of the half-circles  are singular and it follows that the scattering relation is $\CHS_{\Sigma}(p) = -p$.  

To understand the scattering map $F_\Sigma$ we need to relate the orbit parameters $[\gamma]$ at the two endpoints.  These can be found by using the formulas for $s_1, \rho_1$ as noted in the proof of Lemma~\ref{lemma_timereversal}.  

Since $\rho(\tau) = 0$ all along the orbit, we have
$$\rho_1= \lim_{\tau\into \mp\infty} \exp(\pm v_0\tau)\rho(\tau) = 0$$
at both endpoints.  Using this, we have
$$s_1^- = \lim_{\tau\into -\infty} \exp(v_0\tau)(s(\tau)-s_0)\qquad s_1^+ = \lim_{\tau\into \infty} \exp(-v_0\tau)(s(\tau)+s_0).$$
Using the formula for $s(\tau)$ in Proposition~\ref{prop_infinityflow} it is easy to see that $s_1^- = s_1^+ =2\eta$.
\end{proof}

\subsection{Scattering Near Infinity}
What if we ``expand'' $\Sigma$, a bit?   Let $Z = Z(R) $ denote the invariant set consisting of all points $ x $ such that the orbit $ \phi_\tau(x) $ satisfies $\rho(\tau)< 1/R$ for all $\tau$  and let   $Z(R,K)$ be the set whose orbits   also satisfy $U(s(\tau))< K$ for all $\tau$. 
If $K$ is sufficiently large, then there will be a nonempty set of points on the infinity manifold in $Z(R, K)$.
Using the local structure near the equilibria, it is easy to see that $Z(R,K)\cap \CH$ and $Z(R)\cap \CH$ are nonempty open sets of phase space. 
We will denote the restricted scattering maps as $F_R, F_{R,K}$ and the restricted relations by  $\CHS_{R}$ and $\CHS_{R,K}$.  Similarly, we use the notation  $p\rightarrow_{R} q$ and $ p\rightarrow_{R,K} q$.  Note that adding restrictions leads to a smaller scattering relation.  So, for example, $\CHS_{R,K}\subset \CHS_{R} \subset \CHS$. Equivalently, $p\rightarrow_{R,K} q$ implies $p\rightarrow_{R} q$ which implies $p\rightarrow q$. 
 
Recall that the scattering relation at infinity is $\CHS_\Sigma(p) = -p$.  The next two results show that scattering via orbits sufficiently near infinity leads to final states near $-p$.  The first result  also requires a bound on the potential.
\begin{theorem}\label{th_scatteringnearinfinity}
Fix $K>0$ , $p\in\CEh_-$,  $\CV\subset\CEh_+$  a neighborhood of $-p$. Then for all   $R$ sufficiently  large  we have 
$\CHS_{R,K}(p)\subset \CV$. 
\end{theorem}
\begin{remark} If $K < U(p)$ it is possible that $\CHS_{R,K}(p)$ is empty. Nevertheless, the theorem still holds.
\end{remark} 
\begin{proof}
Consider the unstable manifold of $p$ {\it within the infinity manifold} $\hat{\Sigma}$. We can choose a local cross-section $\s$ to the flow  within this manifold which is diffeomorphic to   the  $  D-1 $ dimensional sphere $S$.  Let   $\s_K \subset \s$ be the subset   whose orbits satisfy $U(s(\tau))\le K$ for all $\tau$. The complement of $\s_K$ in $\s$ is open so $\s_K$ is compact.  Now for each point $x\in \s_K$, we have $\pi_+(x) = -p$
since the orbit through $x$ lies on the infinity manifold and limits to $p$ in the backward direction.
  Since $\pi_+$ is continuous,  there is some neighborhood $\CU(x)$ of $x$ in the full phase space such that $y\in\CU(x)$ implies $\pi_+(y)\in\CV$.  Since $\s_K$ is compact, there is some neighborhood $\CU$ of $\s_K$ itself with this property.  We can thicken the local cross section $\s_K$ in $\Sigma =\{\rho=0\}$  to a cross-section of the form $\s_K\times [0,\delta)$ where $u\in [0,\delta)$.  For $\delta$ sufficiently small, we will have $\s_K\times [0,\delta)\subset \CU$.  Then choose $R =1/\delta$.  If $x\in Z(R,K)\cap \CH$ then its orbit meets the cross-section $\s_K\times [0,\delta)$ and
$\pi_+(x)\in\CV$.  Since this holds for all such $x$,  $\CHS_{R,K}(p)\subset \CV$.
\end{proof}

If we want to avoid imposing a fixed upper bound on $U(s)$, we can't guarantee a uniform $1/R$ sufficiently small, but we can still find a neighborhood of the set of nonsingular orbits at infinity starting at $p$ such that scattering via orbits in this neighborhood always leads near $-p$.
\begin{theorem}\label{th_scatteringnearinfinityA}
Let $p\in\CEh_-$ and  let $\CV\subset\CEh_+$ be a neighborhood of $-p$.  Then there is a neighborhood $Z$ in phase space of the set of nonsingular orbits in $W^u(p)\cap \hat{\Sigma}$ such that  $\CHS_{Z}(p)\subset \CV$. 
\end{theorem}
\begin{proof}
Let $Z =\CH \cap \pi_+^{-1}(\CV)$.  Then $Z$ is an open, invariant subset of phase space.  If $x\in\Sigma$ is a point on any nonsingular orbit at infinity with $\pi_-(x) = p$ then $x\in Z$, so $Z$ is a neighborhood of the nonsingular orbits in $W^u(p)\cap \Sigma$.
\end{proof}

Theorems \ref{th_scatteringnearinfinity} and \ref{th_scatteringnearinfinityA} assert a kind of continuity property of the scattering relation when  restricted to neighborhoods of the infinity manifold.  They can also be viewed as putting an upper bound on the image of $p$ under near-infinity scattering.  Next we turn to the more challenging question of lower bounds.  From what has been proved so far, it is conceivable that $\CHS(p) = \{-p\}$, i.e., that $p\rightarrow -p$ is the only scattering possible.  We want results showing that 
$\CHS(p)$ is nontrivial.  Our main result is  the  ``openness'' of $\CHS(p)$ described in property (vii) of theorem \ref{thm:scatteringreln}.   
The proof will involve perturbing a simple half-circle orbit at infinity and will actually show the stronger result that the image set $\CHS_{R,K}(p)$ contains an open set for certain $R,K$.  In other words, by scattering using bi-hyperbolic orbits which remain near infinity and are bounded away from the singularities, we can reach an open set of final states.  First we show that the image contains at least a curve, then we fatten the curve into the required open set.

Fix an equilibrium point  $p=(s_0,v_0)\in \CEh_-$.  
The local unstable manifold of $p$ is parametrized by $(s_0, v_0, \rho_1,s_1)$ as per equation (\ref{eq:alignedCoords}).  
We get the unstable manifold in $\hat{\Sigma}$ in this parameterization by taking $\rho_1=0$. 
Alternatively, the orbit parameter is $[\g]$ where $\g= (0,s_1)$.  
For these orbits, we have seen that  the scattering map is $F(s_0,v_0,[\g]) = (-s_0,-v_0,[\g])$.   
We want to perturb to study orbits with $\rho_1$ small compared to $s_1$.

If we normalize $\normtwo{s_1}=1$, then Proposition~\ref{prop_infinityflow} gives the unstable orbits in $\hat \Sigma$ by setting $\xi = -s_0, \eta = s_1, v_0 = -\sqrt{2h}$. 
Recall that the flow at infinity is really a reparametrization of the free particle flow with $q(t) = At+C$.  Thus $A=\dot q$, but $\dot q= vs +w$
in   the  variables
$(\rho, s, v, w)$ so we have $$A(s,v,w) = vs +w.$$
The following lemma gives  the first-order  change in  $A$ for perturbations of our unstable-manifold orbit in  $\Sigma$.
\begin{lemma}\label{lemma_scatteringformula}
 Consider any nonsingular orbit at infinity, given by Proposition~\ref{prop_infinityflow} with $s_0 = -\xi,  v_0=-\sqrt{2h}, \rho_1=0,s_1=\eta$ and let $(\delta\rho, \delta s,\delta v,\delta w)$ be any solution of the variational differential equation along this orbit.  Then the total change in $A(s,v,w) = vs+w$ along the variational orbit is  given by
 \begin{equation}\label{eq_scatteringformula}
 \Delta A(\xi,v_0,\eta) = \frac{\delta\rho(0)}{\sqrt{2h}}\int_{-\frac\pi{2}}^{\frac\pi{2}}\nabla U(\xi\sin\theta+\eta\cos\theta)\,d\theta.
 \end{equation}
\end{lemma}

 \begin{proof}
 From (\ref{eq_odex}) we find $A' = -\rho U(s)s +\rho \tilde\nabla U(s) = \rho\nabla U(s)$ and the first order variation along an orbit with $\rho=0$ is
 $$(\delta A)' = \nabla U(s(\tau))\delta\rho$$
 where $\delta\rho(\tau)$ satisfies $\delta\rho'= -v\delta\rho$.

 Now we will switch to use the variable $\theta$ from Proposition~\ref{prop_infinityflow} as the independent variable.  We have
 $v(\theta) = \sqrt{2h}\sin\theta, \theta' = \sqrt{2h}\cos\theta$
and so
 $$\frac{d(\delta\rho)}{d\theta} = -\tan\theta\,\delta\rho\qquad \delta\rho(\theta) = \delta\rho(0)\cos\theta.$$
 The variational equation for $\delta A$ becomes
 $$\frac{d(\delta A)}{d\theta} = \frac{\delta\rho(0)}{\sqrt{2h}}\nabla U(s(\theta))$$
 and (\ref{eq_scatteringformula}) follows by integration.
 \end{proof}

For later use, we note that the parameter $\delta\rho(0)$ for a solution of the variational equation is related to the parameter $\rho_1$ in the linearized flow (\ref{eq_stableunstablelinearizedflow})  near $(s_0,v_0)$ by
\begin{equation}\label{eq_rho1}
\rho_1 = \lim_{\tau\into-\infty}\exp(v_0\tau)\delta\rho(\theta(\tau))   =\lim_{\tau\into-\infty}\exp(v_0\tau)\delta\rho(0)\sech(v_0\tau)=2\delta\rho(0).
\end{equation}

The integral in (\ref{eq_scatteringformula}) is difficult to calculate, except for the following special case. 
 
\begin{example}
\label{eg_planarcurve}Consider the planar $n$-body problem and introduce the notation
$s^\perp$ for the configuration where each position vector in $s$ is rotated by $\pi/2$. Then taking $\eta=\xi^\perp$ in Proposition~\ref{prop_infinityflow} gives a path
$$s(\theta) = \xi\sin\theta +\xi^\perp\cos\theta\qquad -\frac\pi{2}\le \theta \le \frac\pi{2}$$
from $s_0=-\xi$ to $s_0'=\xi$ which consist of rigid rotations of $\xi$.  In fact, letting the rotation matrix $R(\theta)= \m{\sin\theta& -\cos\theta\\ \cos\theta&\sin\theta}$ act diagonally on configuration space, we have $s(\theta) = R(\theta)\xi$.  The rotational symmetry of the potential gives $U(R(\theta)\xi)=U(\xi)$ and
$$\nabla U(R(\theta)\xi) = R(\theta)\nabla U(\xi) =\nabla U(\xi)\sin\theta +\nabla U(\xi)^\perp\cos\theta.$$
Then (\ref{eq_scatteringformula}) gives
$$\Delta A(\xi,-\sqrt{2h},\xi^\perp) = \frac{2\delta\rho(0)}{\sqrt{2h}}\nabla U(\xi)^\perp.$$
The same proof works when $s_0$ is a planar configuration in $\R^d$, that is, a configuration such that all of the bodies lie in some fixed plane. Then $R(\theta)$ becomes a rotation of that plane which fixes the orthogonal complement.
\end{example}

This calculation gives some (rather minimal) information about the scattering map and relation.

\begin{proposition}\label{prop_imagecurve}
Let $p=(s_0,-\sqrt{2h})\in \hat\CE_-$ where $s_0$ is a planar configuration.  Then the image $\CHS(p)$ under the hyperbolic scattering relation contains an analytic curve of the form $q(\rho_1) = (s(\rho_1),\sqrt{2h})$, $0\le\rho_1<\delta$ for some $\delta>0$ with $s(0)=-s_0$ and tangent vector $s'(0) = \frac{1}{2h}\nabla U(-s_0)^\perp$.  The same holds for  $\CHS_R(p)$ and $\CHS_{R,K}(p)$ for $K>U(s_0)$.
\end{proposition}
\begin{proof}  Since the planar $n$-body problem sits within the spatial and higher dimensional $n$-body problem
we can view the solution  curve $s(\theta)$ at infinity  of the example above  as occurring within any d-dimensional $n$-body problem, $d\ge 2$.  So,
as  above let $\xi=-s_0, \eta=\xi^\perp$.  Consider the curve in the local unstable manifold with parameters $(\rho_1,s_1)$, $\rho_1\in [0,\delta), s_1=\xi^\perp$ (see Figure~\ref{fig:conenbhdcurve}).   If $\delta>0$ is sufficiently small,  the curve will lie in $\CH$.  Moreover, since $\rho(s(\theta))=0$ and $U(s(\theta)) = U(s_0)$ for all $\theta$, the bi-hyperbolic orbits  starting at $(\rho_1,s_1)$ will satisfy $\rho(\tau)<1/R$ and $U(s(\tau))<K$ for all $\tau$, provided $\delta$ is sufficiently small.

Now consider the image curve $\pi_+(s_0,-\sqrt{2h},\rho_1,\xi^\perp)$.  It's an analytic curve in $\hat\CE_+$ with $\pi_+(s_0,-\sqrt{2h},0,\xi^\perp) = (-s_0,\sqrt{2h})$.  Along this curve we have 
$$A(\pi_+(s_0,-\sqrt{2h},\rho_1,\xi^\perp)) = \sqrt{2h}s(\rho_1)$$
 and, to first order in $\rho_1$,
$$\sqrt{2h}s'(0)\rho_1 =  \Delta A(\xi,-\sqrt{2h},\xi^\perp) = \frac{2\delta\rho(0)}{\sqrt{2h}}\nabla U(\xi)^\perp =  \frac{\rho_1}{\sqrt{2h}}\nabla U(\xi)^\perp$$
which gives the formula for $s'(0)$.
\end{proof}

With more  work we can expand this curve into an open set.
\begin{theorem}\label{thm:openimage}
Let $p=(s_0,-\sqrt{2h})\in \hat\CE_-$ where $s_0$ is a planar, but non-collinear, configuration. Then the image $\CHS(p)$ under the hyperbolic scattering relation contains an open subset of $\{(s,v)\in \hat\CE_+:v=\sqrt{2h}\}$.  The same holds for  $\CHS_R(p)$ and $\CHS_{R,K}(p)$ for $K>U(s_0)$.
\end{theorem}

\begin{proof}
The open subset will be a  neighborhood of the curve in Proposition~\ref{prop_imagecurve} with its initial endpoint  $\rho_1 = 0$ deleted.  In other words, around each point of the curve with $0<\rho_1$ sufficiently small, there is some neighborhood contained in $\CHS(p)$.   

Let $\xi=-s_0, \eta=\xi^\perp$.  Consider the manifold $Z$ of initial conditions in $W^u_{loc}(p)$ of the form
$0\le \rho_1<\delta$, $\normtwo{s_1-\eta}<\delta$ with $\metrictwo{s_0,s_1}=0, \normtwo{s_1}=1$ (see Figure~\ref{fig:conenbhdcurve}).  For $\delta>0$ sufficiently small, we have $Z\subset \CH$ and we can consider the image
 $\pi_+(Z)\subset \CHS(p)\subset \hat\CE_+$.  Points in this  image are  of the form $\pi_+ (\rho_1, s_1) = (s (\rho_1,s_1),\sqrt{2h})$
 for some analytic map $(\rho_1, s_1) \mapsto s(\rho_1, s_1)$ with $s(0,s_1)= -s_0$.    Define 
 $$g(\rho_1,s_1) = 4h s (\rho_1, s_1) = 2\sqrt{2h} A (\rho_1, s_1)$$ 
where $A = v_0 s = \sqrt{2h} s$ is the Chazy parameter $A$ as per Lemma \ref{lemma_scatteringformula}. 
Lemma~\ref{lemma_scatteringformula} gives the  Taylor expansion in $\rho_1$ for $g$ to be
 $$g(\rho_1,s_1) = \rho_1 I(\xi,v_0,s_1) +O(\rho_1^2)$$
 where 
 $$I(\xi,v_0,s_1)=\int_{-\frac\pi{2}}^{\frac\pi{2}}\nabla U(\xi \sin\theta+ s_1\cos\theta)\,d\theta.$$
 We will show that the derivative of $g$ is nondegenerate at points of the form $(\rho_1,\eta)$, $\rho_1>0$ sufficiently small.  Then the inverse function theorem shows that $g(Z)$ contains a neighborhood of $g(\rho_1,\eta)$. Consequently we get a neighborhood of our curve, as claimed.

\begin{figure}[ht]
\scalebox{0.6}{\includegraphics{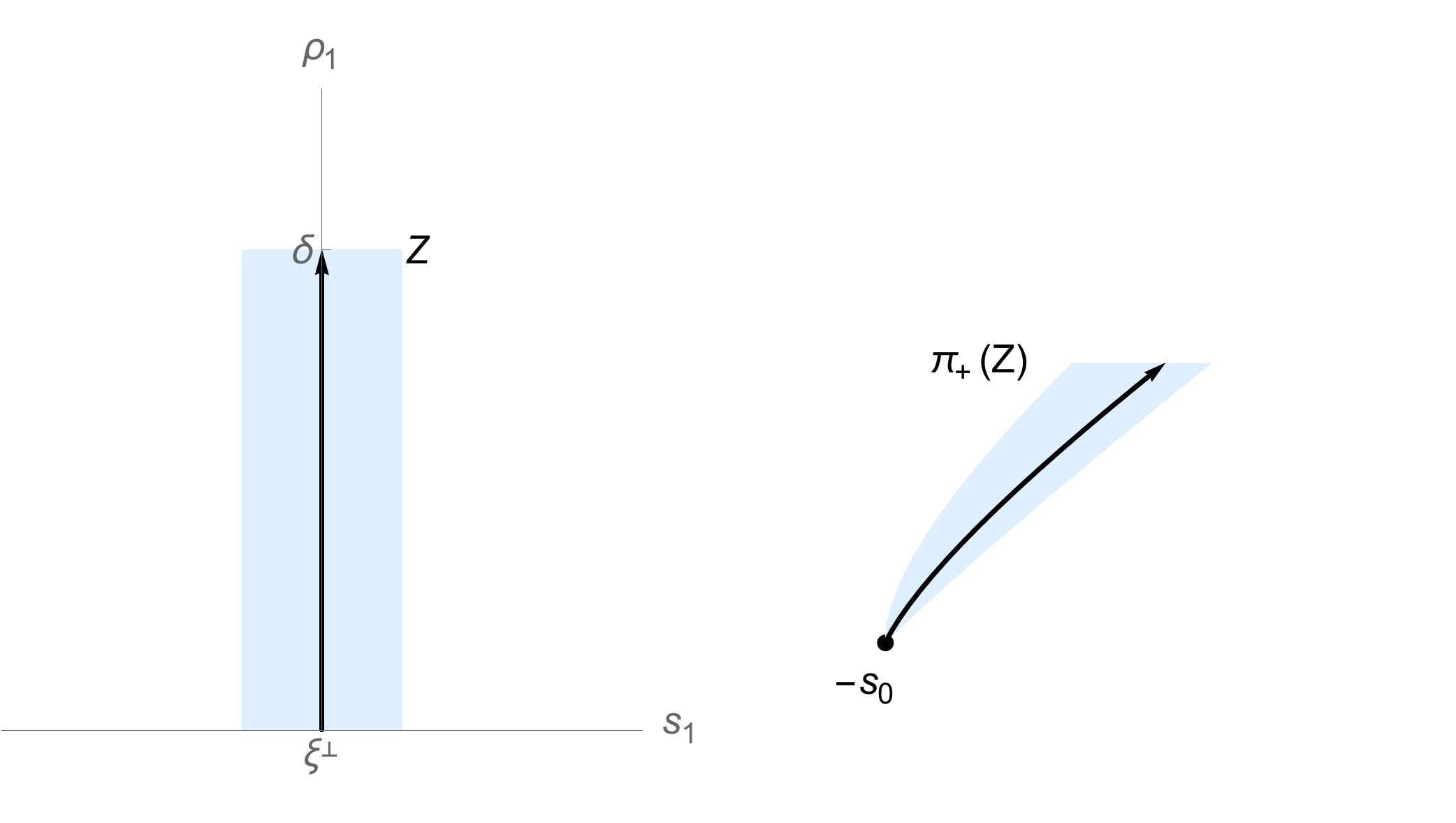}}
\caption{The domain $Z$ parametrizing the thickened curve and its image, the thickened curve.  $Z\subset W^u_{loc}(s_0,v_0)$ and the image is contained in $\CE_+$.}
\label{fig:conenbhdcurve} 
\end{figure}
 
 We have
 $$g_{\rho_1}(\rho_1,\eta) = I(\xi,v_0,\eta) +O(\rho_1) = 2\nabla U(\xi)^\perp+O(\rho_1)$$
 and 
 $$\begin{aligned}
 g_{s_1}(\rho_1,\eta) &= \rho_1\frac\partial{\partial s_1}I(\xi,v_0,s_1)|_{s_1=\eta} +O(\rho_1^2)\\
&= \rho_1 \int_{-\frac\pi{2}}^{\frac\pi{2}}D\nabla U(\xi\sin\theta+\eta\cos\theta)\,\cos\theta\,d\theta +O(\rho_1^2).
\end{aligned}
 $$
 
 {\bf Case $d=2$.}  We first give the proof for the case of the planar $n$-body problem and later extend to the case of a planar configuration in $\R^d$.
Then the  right hand side of this last  formula  is a $2n \times 2n$ matrix which we will now evaluate.  The rotational symmetry of the potential gives
$$D\nabla U(\xi\sin\theta+\eta\cos\theta) = D\nabla U(R(\theta)\xi) = R(\theta)D\nabla U(\xi)R(\theta)^{-1}$$
where  $R(\theta)$ acts on the $2\times 2$ blocks of $D\nabla U(\xi)$. 

It is straightforward to calculate the $2n\times 2n$ matrix $D\nabla U(\xi)$ with the result
\begin{equation}\label{eq_DnablaU}
D\nabla U(\xi) = \m{D_{11}&D_{12}&\ldots&D_{1n}\\
D_{21}&D_{12}&\ldots&D_{2n}\\
\vdots&&&\vdots}
\end{equation}
where the $2\times 2$ blocks are
$$D_{ij}  = \frac{m_j}{r_{ij}^3}\left(I - 3u_{ij}u_{ij}^T\right), \quad u_{ij} = \frac{\xi_i-\xi_j}{r_{ij}}\qquad \text{ for }i\ne j$$
and 
$$D_{ii} = -\sum_{j\ne i}D_{ij}$$
where $r_{ij}=|\xi_i-\xi_j|$.
The action of the rotation just conjugates each $2\times 2$ block $D_{ij}$ by $R(\theta)$.  One can check that the result of conjugating each block by $R(\theta)$, multiplying by $\cos\theta$ and then integrating over $[-\pi/2,\pi/2]$ is to replace the blocks $D_{ij}, D_{ii}$ by
$$\bar{D}_{ij} = \frac{-2m_j}{ r_{ij}^3}v_{ij}v_{ij}^T\qquad \bar{D}_{ii} = -\sum_{j\ne i}\bar{D}_{ij}$$
where $v_{ij} = u_{ij}^\perp$.
Writing  $\bar D$ for this integrated matrix, we have
 $$g_{s_1}(\rho_1,\eta) =\rho_1\bar D +O(\rho_1^2).$$

 If $\beta\in\R^{2n}$ then one can check that
$$\beta^T M \bar D \beta = 2\sum_{i<j}\frac{m_im_j}{r_{ij}^3}(\beta_{ij}\cdot v_{ij})^2.$$
Regarding the null space of this quadratic form, which is the same as the kernel of $\bar D$, we have the following lemma, whose proof will be postponed for the moment.

\begin{lemma}
\label{lem: xi nondeg} The kernel of $\bar{D}$ has dimension $3$ and can be spanned by three independent vectors
$\beta  = \xi$, $\beta = (1,0,1,0 \dots)$ and $\beta = (0,1,0,1,\ldots)$.
\end{lemma}

Now we  should  restrict  $g_{s_1}(\rho_1,\eta) $ and $\bar D$ to vectors $\beta$ with center of mass 0 and with $\metrictwo{\beta,s_0}=\metrictwo{\beta,\eta}=0$ since, in $Z$,  we are perturbing away from $s_1=\eta$ while keeping $\metrictwo{s_0,s_1}=0$ and $\normtwo{s_1}=1$.  Lemma \ref{lem: xi nondeg}  shows that none of these vectors are in  the kernel of $\bar D$.  On the other hand, the vector
$$g_{\rho_1}(\rho_1,\eta) =   2\nabla U(\xi)^\perp+O(\rho_1)$$
and this is equal to $\bar D \eta$.    Namely, the $i$-th pair of components of $\bar D \eta$ is
$$(\bar D \eta)_i = (\bar D\xi^\perp)_i= 2 \sum_{j\ne i}\frac{m_j}{r_{ij}^3}(\xi^\perp_i-\xi^\perp_j) = (2 \nabla U(\xi)^\perp)_i.$$
Lemma \ref{lem: xi nondeg} shows that $\eta$ is also not in the kernel of $\bar D$.  Therefore, 
taking the vector $\bar D \eta$ together with the image under $g_{s_1}(\rho_1,\eta) $ of all relavant vectors $\beta$, we find  that for $\rho_1$ sufficiently small, the image of the derivative map of $g$ at $(\rho_1,\eta)$ has dimension $2n-3 = D-1$.  Since this is the dimension of $Z$ and of $\{(s,v)\in \hat\CE_+:v=\sqrt{2h}\}$, the derivative map  is indeed nondegenerate for $\rho_1$ sufficiently small and the proof is complete.

 {\bf Case $d>2$.}    When $d > 2$ we think of the planar $ n $-body problem as it embeds
 in the higher-dimensional $ n $-body problem, writing $\R^d = \R^2 \oplus \R^{d-2}$ as
 an orthogonal decomposition.   We continue to perturb off the same analytic
 curve of hyperbolic solutions  as per proposition \ref{prop_imagecurve}, with the inducing hyperbolic
 scattering curve now  viewed as lying in $(\R^2)^n \subset (\R^d)^n$.  To show that the derivative along this curve is
 invertible we  split each component variation $\beta_i$ into $\beta_i = \beta_i ^{||} + \beta_i ^{\perp}$, $\beta_i ^{||} \in \R^2, 
\beta_i^{\perp} \in \R^{d-2}$.
 The formulae for $D_{ij}$ still holds. 
 Since  $u_{ij} \in \R^2$ so we have that each  $D_{ij}$ itself  splits into  block-diagonal form : $D_{ij} \beta_j = D_{ij}^{||} \beta_j^{||} + D_{ij}^{\perp} \beta_j ^{\perp}$
 with $D_{ij}^{||} : \R^2 \to \R^2$  being exactly the same as the 2 by 2 block
 $D_{ij}$ for the $d =2$ case, and with  $D_{ij}^{\perp} : \R^{d-2} \to \R^{d-2}$ equal to  $(m_j / r_{ij}^3 ) I_{d-2}$, a multiple of the identity independent of $\theta$. 
 The conjugation by $R(\theta)$ result still holds, with $R(\theta)$ acting trivially on the orthogonal ``$\perp$'' blocks.  Multiplying by $\cos\theta$ and integrating gives a block matrix $\bar D$ with a $2\times 2$ block ${\bar D}^{||}$ as above and a $(d-2)\times (d-2)$ block ${\bar D}^{\perp}$ which is just the original block $D^\perp$ multiplied by a factor of 2.
 
 The end result is that
 $$\beta^T M \bar D \beta =2\sum_{i<j}\frac{m_im_j}{r_{ij}^3}(\beta_{ij}^{||}\cdot v_{ij})^2 - 2 \sum_{i<j}\frac{m_im_j}{r_{ij}^3}|\beta_i ^{\perp} -\beta_j ^{\perp} |^2.$$
Recall that Lemma \ref{lem: xi nondeg} implies, the first term, $\bar D^{||}$ has corank 3. The second term has a $d-2$-dimensional kernel consisting of vectors with all $\beta_i ^{\perp}$ equal. It is nondegenerate on the space of vectors $\beta^\perp$ with center of mass zero.
Then  for $\rho_1$ sufficiently small, the image of the derivative map of $g$ at $(\rho_1,\eta)$ has dimension
$dn-3 - (d-2) = d(n-1)-1 =D-1$ and the implicit function theorem applies as before to complete the proof.
 \end{proof}

\begin{proof}[Proof of Lemma \ref{lem: xi nondeg}]

Let $\beta=(\beta_i)_{i=1}^n$ be a vector from the kernel of $\bar D$. Then $\langle v_{ij}, \beta_{ij} \rangle=0$, for all pairs of $i, j$ with $i \ne j$, where $\beta_{ij}= \beta_i - \beta_j$. Recall that $v_{ij} = u_{ij}^{\perp}$ and $u_{ij}= \frac{\xi_i - \xi_j}{r_{ij}}$. This means either $\beta_{ij}=0$ or $\beta_{ij} \ne 0$ and is parallel to $\xi_{ij}=\xi_i -\xi_j$. Notice that $\xi_{ij} \ne 0$, for any $i \ne j$, as $\xi$ is a non-singular configuration.

We claim if there exists a pair of $i \ne j$, such that $\beta_{ij}=0$, then $\beta_{k}=\beta_i$, for all $1 \le k \le n$.  In  this case, $\beta$ is in the span of $(1,0,\dots), (0,1, \ldots)$.
To prove the claim, choose a $k \ne i$ or $j$. First let's assume $\xi_k \notin \overline{\xi_{ij}}$, where $\overline{\xi_{ij}}$ represents the straight line pass through $\xi_i$ and $\xi_j$. Since $\xi$ is non-singular, $\xi_i \ne \xi_j$, and this line is unique.  The non-singularity of $\xi$ further implies $\xi_{ik} \ne \xi_{jk}$ and neither is zero.  Meanwhile $\beta_{ij}=0$ implies $\beta_{ik} = \beta_{jk}$. This means $\beta_{ik}$ must be zero, as otherwise, $\beta_{ik}$ is parallel to both $\xi_{ik}$ and $\xi_{jk}$, which is absurd.

Now assume $\xi_k \in \overline{\xi_{ij}}$. Since $\xi$ is non-collinear, there must exist another index $\ell$, such such $\xi_{\ell} \notin \overline{\xi_{ij}}$. By the previous result, $\beta_{i\ell}=0$. Meanwhile $\xi_k \notin \overline{\xi_{i\ell}}$ and by repeating the previous argument, we can show $\beta_{ik}=0$. This finished our proof of the claim. 

With the above claim, we may assume $\beta_{ij} \ne 0$, for all pairs of $i \ne j$, as otherwise all $\beta_{ij}$ must be zero. Since $\xi$ is non-collinear, there exist at least three different indices $i, j$ and $k$, such that $\xi_i, \xi_j$ and $\xi_k$ are the vertices of a non-degenerate triangle. Meanwhile by the assumption, $\beta_i, \beta_j$ and $\beta_k$ are the vertices of a non-degenerate triangle as well. Since $\beta$ is in the kernel of $\bar D$, the corresponding sides of the two triangles are parallel to each other. Therefore the two triangles must be similar to each other and there is a constant $\lmd >0$, such that 
\begin{equation}
\label{eq_triangle lmd}  \beta_{ij} = \lmd \xi_{ij}, \;\; \beta_{jk} = \lmd \xi_{jk}, \;\; \beta_{ki} = \lmd \xi_{ki}.
\end{equation}
Hence after translating $\beta$ by a constant vector, we have 
\begin{equation}
\label{eq_homoth lmd} 
\beta_i = \lmd \xi_i, \;\; \beta_j = \lmd \xi_j, \;\; \beta_k = \lmd \xi_k.
\end{equation} 

Now for any index $\ell$ different from $i, j, k$, we can always find two vectors from $\xi_i, \xi_j$ and $\xi_k$, such that together with $\xi_{\ell}$, they form the vertices of a non-degenerate triangle. Without loss of generality, we assume $\xi_i$ and $\xi_j$ are such two vectors. Then by a similar argument as above we can show the two non-degenerate triangles form by $\xi_i, \xi_j, \xi_{\ell}$ and $\beta_i, \beta_j, \beta_{\ell}$ correspondingly are similarly to each other, and there is a constant $\mu >0$, such that 
$$ \beta_{ij} = \mu \xi_{ij}, \;\; \beta_{j \ell} = \mu \xi_{j \ell}, \;\; \beta_{\ell i} = \mu \xi_{\ell i}. $$ 
Combining the first identity above with the first identity in \eqref{eq_triangle lmd}, we get $\mu = \lmd$. Then the last identity above implies 
$$ \beta_{\ell}- \beta_i = \lmd \xi_{\ell} - \lmd \xi_{i}. $$ 
Combining this with the first identity in \eqref{eq_homoth lmd}, we get $\beta_{\ell} = \lmd \beta_{\ell}$. As a result we have proven that up to a translation $\beta = \lmd \xi$.  This means that $\beta$ is in the span of the three vectors given in the lemma.
\end{proof}

{\sc Finishing off the proof of Theorem \ref{thm:scatteringreln}.} To finish off the proof of the openness property (vii) of
Theorem \ref{thm:scatteringreln} we must  extend the  results just achieved  for planar configurations to generic non-planar ones when $d>2$.
We will use  an analyticity argument.
\begin{theorem}\label{thm:generic}
There are real-valued continuous functions $f, K$  defined on an open, dense subset of $S$ with $f$ real analytic and non-constant, such that if $p=(s_0,-v_0) \in \hat{\CE}_-$ and $f(s_0)\ne 0$,  then the subsets $\CHS(p)$, $\CHS_R(p)$ and $\CHS_{R,K}(p)$ for $K>K(s_0)$  all satisfy the openness condition,
condition  (vii) of theorem \ref{thm:scatteringreln}: each set contains a non-empty open
subset of the sphere $\hat{\CE}_+ \cap \{v = + v_0\} $ whose closure contains $-p$.
\end{theorem}
\begin{proof}
 When  $s_0$ is planar and non-collinear, the  proofs we  just gave  above  that the sets $\CHS(p), \ldots $ satisfy condition (vii) boiled
down to showing  a certain linear operator had full rank, this rank being $D-1$.  This operator was the restriction of 
   $\bar D$  to the $D-1$-dimensional linear space  $V(s_0)$ of vectors with zero center of mass and  orthogonal to $s_0$.  
Once bases are chosen  in the domain and range of the operator,  this rank condition  could have been verified
 by showing a    determinant is nonzero.   When  $s_0$ is
not planar, the linear space $V(s_0)$ still  has dimension $D-1$,   the same as the dimension of the target space
 $\{(s,v)\in \hat\CE_+:v=\sqrt{2h}\} \cong S$.  Nondegeneracy  of a  corresponding operator  will  allow us to apply the inverse function theorem and
conclude that the  sets again satisfy condition (vii).

$\bar D$ was defined as the integral  
$$\bar D(s_0) = \int_{-\frac\pi{2}}^{\frac\pi{2}}D\nabla U(\xi\sin\theta+\eta\cos\theta)\,\cos\theta\,d\theta $$
where $\xi = - s_0$ and $\eta = \xi^\perp$ (obtained by rotating all of the $\xi_i$ by 90 degrees in the plane).  This  construction of
an operator can be extended to almost all nonplanar $s_0$ by a choosing an $\eta = \eta(s_0)$  in the following (rather artificial) way.

Let $\xi = -s_0$ where $s_0$ is not necessarily planar.  Project all of the $\xi_i$ into $\R^2 \times 0$ and let $\eta$ be obtained by rotating all of these projections in that plane.  As long as not all of the projections are $(0,0)$ this gives a nonzero vector $\eta$ which is orthogonal to $\xi$ and  which can be normalized  so  $|| \eta || = 1$ as in  Proposition~\ref{prop_infinityflow}.  Call this vector $\eta(s_0)$.  This $\eta(s_0)$ is well defined and analytic on an open and dense subset of $\CE_-$.  We define  $\bar D(s_0)$ to be the integral above, an $nd \times nd$ matrix depending analytically on $s_0$, 
and defined for all $s_0$ in the open, dense set $\CU$ of $s_0$'s satisfying the projection condition above and the additional condition that   that the great circle  path $s(\theta,s_0)= -s_0\sin \theta + \eta(s_0) \cos \theta$ misses the collision locus.  The operator continues to represent
the linearized flow along the corresponding great circle, `pushed' out infinitesimally into the `bulk' $\rho> 0$ so the arguments relating
the nondegeneracy of $\bar D(s_0)$ to condition (vii) continue to hold. 
To see that $\CU$  can  be specified by the nonvanishing of an analytic function $g$, observe that the projection condition, being a linear independence condition, can be written as the non-vanishing
of   products of minors built from the analytic vector functions   $-s_0$ and $\eta(s_0)$,
while  the non-collision condition can be written as the nonvanishing of the pullback of the product of the algebraic functions  $r_{ij} ^2$'s by
the geodesic flow on the sphere, which is analytic.   

Regarding the function $K$, note that  although our paths  $s(\theta,s_0)$ avoid the poles of $U$,  the potential $U(s(\theta,s_0)))$ is no longer constant along them.  For this reason, the constant $U(s_0)$ in the planar theorem should be replaced by $K(s_0) = \max_\theta U(s(\theta,s_0)))$.

Now let $\CS(s_0)$ be the $D-1$  dimensional subspace of vectors with zero center of mass and orthogonal to $s_0$.
If we can choose a basis for $\CS(s_0)$ in an analytic way, then the nondegeneracy becomes the requirement that some $(D-1) \times (D-1)$ determinant  be nonzero.  Multiply this by  determinant by $g$  we get the required $f$.
There is no problem choosing a basis for the zero center of mass subspace (for example the basis implicit in Jacobi coordinates).  Choosing a basis for the spaces orthogonal to $s_0$ is not possible globally on $S$ (unless $D-1 =3$ or $7$)  but since we have deleted the singular points, the resulting  normal bundle is trivial and such a basis can be found.  We now have our function $f$. 

We now have condition (vii) holding for the set $\CV$  all $s_0$'s such that   $f(s_0)\ne 0$.  $\CV$ is nonempty
because it contains the planar nonconlinear $s_0$'s described earlier. $\CV$ is clearly open. 
To see that $\CV$ is dense, we will use the fact that the open, dense set where $g(s_0)\ne 0$ is   connected.  To see its connectedness, 
note that the projection condition gives a connected set since we can use a homotopy to carry out the projection (without changing $\eta$).  Now the set of initial conditions $\xi,\eta$ leading to collision  paths $\xi \sin \theta + \eta \cos \theta$ has codimension at least two, so it can be avoided.
Thus any configuration $s_0$  with $g(s_0)\ne 0$ can be connected to a planar, noncollinear one by an analytic curve, say $\gamma(s)$, with $G(\gamma(s))\ne 0$.  Then $f(\gamma(s))$ is an analytic function of one variable which is not identically zero.  Therefore its zeros are isolated.
\end{proof}

{\sc What's Left to do? } Most   questions  regarding  hyperbolic scattering   remain open.  The broadest  is simply
\begin{problem}
From a given initial state $p\in \CEh_-$ determine which final states in $\CEh_+$ can be reached by hyperbolic scattering.  In other words, determine the image $\CHS(p)$ under the hyperbolic scattering relation.
\end{problem}
For the two-body problem in the plane we saw in example \ref{ex_Kepler}  that every state can in $\CE_+$ having the same energy
as $p$ can be reached, except for $-p$.  If the infinity manifold is added, then its flow takes $p$ to $-p$. Written
in terms of Chazy variables, $\CHS( A) = S^1$, the full circle of asymptotic states of radius $\|A \|$.   
Perhaps an analogous result holds for the $ n $-body problem, i.e. it could be that $\CHS(A) = \Sh$ for all $A \in \Sh$.

 Since so little is known about the problem with $n\ge 3$, even the following seems to be mostly open.
\begin{problem}
From a given initial state $p\in \CEh_-$ give examples of states besides $-p$  which can be reached or of states which cannot be reached by hyperbolic scattering.
\end{problem}

One can pose the same problems with some restrictions on the scattering relation, for example, with $\CHS(p)$ replaced by $\CHS_{R}(p)$ or $\CHS_{R,K}(p)$.   Theorem~\ref{th_scatteringnearinfinity} shows that certain states, namely those outside $\CV$, cannot be reached using $\CHS_{R,K}(p)$, while Proposition~\ref{prop_imagecurve} and Theorems \ref{thm:openimage}  and \ref{thm:generic} give some information in the other direction.  But little else seems to be known.

It seems likely that near collisions will lead to very different kinds of scattering, even for orbits near infinity. In particular, the scattering relation $\CHS_R$ does not impose a bound on the potential and allows near collisions.
\begin{problem}
From a given initial state $p\in \CEh_-$, determine the image $\CHS_R(p)$.  Calculate $\cap_{R>0}\CHS_R(p)$.  Is it equal to $\{-p\}$ ?
\end{problem}

Ideas from \cite{Maderna}  might help in  making progress on these problems. 

\color{black}

\appendix

\section{Analytic Linearizations}\label{sec_ProofOfThm}

Here we prove the analytic linearization theorem, Theorem \ref{th_analyticlinearizationv1} as a  
corollary to:   

\begin{theorem}\label{th_analyticlinearization}
	Let $ X$ be a real analytic vector field defined   on an open subset  of 
	 $\V = \R^m \times \R^k$ and vanishing on $\CE = U \times 0$ 	 where  $U$ is an open subset of  $\R^m$.  
	 Suppose    that    the   linearization
	 $L(p) = DX_p: \V \to \V$ of $X$ at  each $p \in \CE$    has exactly one   nonzero eigenvalue $\lambda(p)$ (necessarily real since $X$ is real)  of   algebraic   	  multiplicity $k$.  Set $N_p = L(p) (\V) \subset \V$ and write  $N \to \CE$ for the rank k analytic vector bundle over $\CE$ whose fiber
	 over $p$ is  $N_p$.  Let $X_N:  N \to N$ be the fiber-linear vector field given by  $X_N (p, v) = (0, L(p) v)$. 
	  Then there is an analytic diffeomorphism $\Phi$ such that $\Phi^* X = X_N$
	  where $\Phi$ maps a   \nbhd of the zero section in $N$ to  a neighborhood of $\CE$ in $\V$.  
	  Finally, if we define  $i: N \to  \V$  by $i(p,v) = p+ v$,  then,  upon restriction to a   \nbhd of the zero section,  $i$ is an   analytic  diffeomorphism onto a \nbhd of $\CE$   and $\Phi$ agrees with $i$ to 1st order along the zero section.
	 
	 \end{theorem}

\begin{remark} The kernel of $L(p)$ equals $\R^m \times 0$. This fact follows   from the assumptions that   $U$ is open
and that the  multiplicity of $\lambda(p)$ is $k$.  
$N_p$   is a complementary subspace  to  $\R^m \times 0$   but need not be  equal to $0 \times \R^k$.
$L(p)$ restricted to $N_p$ is a rank k linear map. It is essential for our application, namely theorem  \ref{th_analyticlinearizationv1}, that 
this map be  allowed to have nonzero Jordan blocks.

\end{remark} 

Let us see   see how theorem \ref{th_analyticlinearizationv1} follows from theorem \ref{th_analyticlinearization}.  

\begin{proof} [Proof of theorem \ref{th_analyticlinearizationv1}]. 
Let $X^{newt}$ be   our N-body vector field, extended to infinity, as per \eqref{eq_odex}.
We saw in the discussion of the linearization structure,  immediately  following \eqref{eq:tgtSpace} and  also   remark \ref{rmk:Alert},  that $X^{newt}$ 
is the restriction to a codimension 2 analytic subvariety of an analytic vector field $X$
 defined on an open subset of $\R \times \E \times \R \times \E = \V$ and whose variables we
wrote as $(\rho, s, v, w)$ in eq \eqref{eq_odex}. Rearranging these  coordinates in the order $(s, v, \rho, w)$,  
viewing $(s,v) \in \R^m, (\rho, w) \in \R^k$ (with $k = m$), and referring to the discussion of the linearization structure 
of $X$ following \eqref{eq:tgtSpace}, we see that $X$ satisfies the hypothesis of theorem
\ref{th_analyticlinearization} with $U = \{(s,v) \in \E \times \R: s \notin \Delta,  v \ne 0 \}$, and with $\lambda(s, v) = -v$.
Theorem \ref{th_analyticlinearization} now supplies the   analytic conjugation to $X_N$ for $X$.    
But  theorem \ref{th_analyticlinearizationv1}  claims the analytic  conjugation  for $X^{newt}$ not  for $X$.  To finish off,
observe that the extended phase space $P$  on which $X^{newt}$ lives is invariant under the flow of $X$, 
and hence $\Phi^{-1} (P)$ is invariant under the flow of $X_N$.    Restricting $X_N$ and   $\Phi$ to $\Phi^{-1}(P)$
yields the claimed conjugacy.
\end{proof}

\begin{proof} 
[Proof of   \ref{th_analyticlinearization}.]

Decompose $\CE$ into $\CE_+$ and $\CE_-$ with  $\CE_- = \{ p  \in \CE : \lambda (p) < 0 \}$
and  $\CE_+ = \{ p  \in \CE : \lambda (p) > 0 \}$.  (The $+$ and $-$ subscripts are for forward-time attracting
and backward time attracting.)  We give the proof for the open component  $\CE_+$ of $\CE$.  
The proof for  $\CE_-$   follows a symmetrical proof with stable manifold replaced by unstable
and $t \to + \infty$ by $t \to - \infty$.  

The proof proceeds in three steps.  {\bf Step 1.}  Straighten out the local stable manifolds so they are aligned with the fibers of $N$.
{\bf Step 2.}  Work within each fiber $N_p$ and linearize the vector field there insuring that the linearization is also analytic in $p$.
{\bf Step 3.}  Use uniqueness of the linearization from step 1 and 2 to insure that the linearizing transformation is defined in a whole 
\nbhd of $\CE_+$. 

{\bf Step 1.}  [Straightening out the stable manifolds.]  The generalized eigenspace for $\lambda(p)$ is $N_p = L(p)( \V)$
and must be the tangent space at $p$  to the local stable manifold $W(p): = W^s _{loc} (p)$
passing through   $p \in \CE_+$. 
Now $\V = \R^m \oplus N_p$.  
 According to the usual stable manifold theorem,    $W(p)$
is the graph of a    smooth  function $\psi_p: (N_p,0) \to \R^m$ with $\psi_p (0) = p$ and $d \psi_p (0) = 0$.  Here, and in the remainder of this section,
the broken arrow notation is used to  indicate that the domain of the function is an open \nbhd of $p$ and not all of $N_p$.

We need  to upgrade $v \to \psi_p (v) $ so as  to be analytic  not just smooth, and    analytic  in both $v$ and  $p$.
To this end, look at  \cite[Chapter 13, Theorem~4.1, also  Exercise~4.11]{Coddington},   where this work is essentially done. 
The  Perron-style proof of the stable manifold  theorem there  proceeds by constructing 
a  nonlinear integral operator  $F= F_{p, v} $ acting on paths.  The path space it acts on is the
space of  smooth paths $\gamma: [0, \infty) \to \V$ which tend to $p$ as $t \to \infty$ and 
have  initial condition $\gamma(0)$ projecting onto $v \in N_p$.  One proves  that for  $|v|$  small    $F$ is
a contraction mapping. Its  unique  fixed point $\gamma(t; p, v)$ lies in   $W(p)$. 
One has, by definition, that   $\gamma(0; p, v) = p + \psi_p (v) + v$ - yielding the map  $\psi_p: \V \to \R^m$. 
To get   $\psi$   analytic in both $p$ and $v$,   complexify both vectors so that   $p \in \C^m$ ,  $v \in  \N_p ^{\C} = N_p \otimes \C$,  and $p + v \in \V^{\C}$.
Write $Re(p), Im(p) \in \R^m,  Re(v), Im(v) \in N_p$ for their real and imaginary parts.
One verifies that all the properties of $F$ are retained in this complexified setting provided   $|Im(p)| + |Im(v)| < \delta$ is sufficiently small.
One also verifies that the iteration scheme  $\gamma^j (t;  p, v) = F_{(p, v)} ^j (\gamma_0 (t)) \in \V^{\C}$  satisfies   uniform
$C^0$ bounds in this thickened \nbhd  with the $\gamma^j$ analytic in  $t, p, v$ at each step. 
Since   the uniform  limit of  complex  analytic functions is complex analytic, we get that the endpoints,
the $\gamma(0; p, v) = p + \psi_p (v)+ v $ are complex  analytic in $p, v$ in this thickened strip, so, upon
restriction to the real parts, are analytic.  We have our analytic stable manifold,
depending analytically on $p$.  

Now write $\phi_p (v) = p + v + \psi_p (v)$ for $v \in N_p$.  Because $W(p)$ is the graph of $\psi_p$
we have that $\phi_p$ maps a \nbhd of 0 in $N_p$  to a \nbhd of $p$ in   $W(p)$. 
Then  $\Phi: N - \to \V$ by $\Phi(p, v) = \phi_p (v)$  is our desired analytic
 straightening
diffeomorphism, with domain a \nbhd of the zero section of $N$.  Recall that  $N_p$ is tangent to $W(p)$ at $p$, 
and that $\psi_p (0) = 0, d \psi_p (0) = 0$ to  see that   the derivative
of $\Phi$  along the zero section $(p,0)$ of $N$ can be identified with the ``identity'' 
$(h, v)  \mapsto  h+v$ from $\R^m \oplus N_p  \to \V$.
In other words, the derivative is invertible and $\Phi$ agrees with   $i$ to first order along the zero section. (
By the inverse function theorem,  $\Phi$ and $i$ are invertible in a \nbhd of the zero section.) 
  Now  the vector field $X$ is everywhere tangent to the foliation
of a \nbhd of $\CE_+$ by the stable manifolds,  the $W(p)$'s, 
so $\Phi^* X$ must be everywhere tangent to the inverse image of this foliation by $\Phi$,
which is to say,  tangent to the fibers of $N \to \CE_+$.  This means that $\Phi^* X$ is a vertical vector field.
Finally, due to the nature of the linearization of $\Phi$ along the zero section, the linearization of 
$\Phi^* X$ and of $X$ both agree along the zero section, which means that
$$\Phi^* X (p, v) = (0 , L(p) v + g(p, v)),  g(p, v) = O(|v|^2) $$
with $g$ analytic.

{\bf Step 2.} [Linearizing fiber-by-fiber].  Invoke the theorem of  Brushlinskaya. 
\begin{theorem}[Brushlinskaya \cite{Brushlinskaya}]\label{thm:ParameterAnalyticConjugacy}
	Consider a family of analytic vector fields $ X_p $ on $ \R^k $, analytically depending on the parameters $ p\in\Omega \subset \R^m$,
	$\Omega$ open.  Assume that at $ 0 $ and $ p=p_0 \in \Omega $ there is an equilibrium with eigenvalues whose real parts  all have  the same sign. Then, with the aid of a near identity transformation $ \Phi$, analytic in a  neighborhood of $ 0 $ and  depending analytically  on the parameters $ p $ within  a  small neighborhood of $ p_0 $, 
	one can  reduce the family $X_p$ to ``resonant polynomial form'':  $\hat X_p = \Phi^* X_p$ is a   family  of polynomial vector fields whose coefficients depend analytically on the parameters $ p $ and for which  the only  monomials occuring in these polynomials
	are the resonant monomials in the sense of the Poincar\'e-Dulac theory.  
	\end{theorem}

	To use this theorem in our situation we will  need to understand the terms  ``resonant monomials'' and ``near-identity''.
	Fix $ p $. Use the notation $ v^Q = v_1^{q_1}\dots v_k^{q_k}, Q=(q_1,\dots,q_k)\in\N^k $ to describe monomials on
	 $\R^k$. Let $\lambda = (\lambda_1,\dots,\lambda_k) $ be the eigenvalues of the linearization of $X_p $ at $0$. 
	\begin{definition} \label{def:resonant}A  resonant monomial $v^Q$  for $X_p$ is a monomial of degree $2$ or higher for which   $Q\cdot \lambda = \lambda_i$ for some $ i = 1,\dots,k$.
	\end{definition}
	\noindent Observe ``degree 2'' or higher is equivalent to  $ |Q| : = q_1+\dots + q_k \geq 2 $.  
	\begin{definition}A near-identity analytic  transformation
	$\phi$ of $\R^k$ is  an analytic transformation defined near $0$
	and having   convergent power series expansion  $v \mapsto v + \phi_2 (v) + \phi_3 (v) + \ldots$,  with  the $\phi_i (v)$ being  homogeneous vector-valued 
	polynomials of degree $i$.  
	\end{definition} 
	
	In the Poincar\'e-Dulac method, as explained for example in 
	 \cite{Arnold}, one tries to  successively kill all the monomials of degree 2 or higher arising in $X_p$ by appropriately choosing the $q_i$ of $\phi$.
	 All non-resonant  monomials can be killed.  The non-resonant ones, being in the kernel of the ``cohomological operator'' 
	 associated to the process remain.  (See step 3 below for this cohomological operator.) The essence of the Brushlinskaya theorem then is that this process leads to a convergent power series
	 for the map $\phi$ with coefficents depending analytically on $p$.  
	 
	 What does that mean for our specific family? We have  that the vector of eigenvectors  $ \lambda$
	 is $\lambda =  \lambda(p)(1, 1, \ldots, 1)$ at   each point $p$.
	 The resonance condition of definition \ref{def:resonant}  then   reads $|Q| \lambda(p) = \lambda(p)$ which is impossible since $|Q| \ge 2$. 
	 There are no  resonant monomials!  The normal form   $ \hat{X}_p $ is simply  the linearization $v \mapsto L(p) v$ of $ \tilde{X}_p $ for each $ p $. 
	 We now have that $X$ is locally  analytically conjugate to its linearizations $X_N$  in  neighborhoods $ U_p  $ of  each $p \in \CE_+ $.
	 
	 {\bf  Step 3.} [Patching together] We must make  sure that all  these local analytic conjugacies guaranteed by the last two steps 
	  piece together to form a single analytic conjugacy defined in a neighbhorhood of  all of $\CE_+$. 
	 
	 To this end, suppose that $\Psi_1, \Psi_2$ are two linearizations, each defined on its  own  open set $\CU_i \subset N$,
	 this open set containing neighborhoods $U_i \subset \CE$.  For simplicity we will say ``$\Psi_i$ is defined over $U_i$'' to encode
	 this information.   Then $\Psi_1 ^* X = X_N$ holds   over $U_1$
	 while $\Psi_2 ^* X = X_N$ holds over $U_2$.  The map $\Phi= \Psi_2 ^{-1} \circ \Psi_1$ is then a near
	 identity transformation defined over $U_1 \cap U_2$ and satisfying $\Phi^* X_N = X_N$ over $U_1 \cap U_2$. 
	 $\Phi$ must preserve the fibers of $N \to \CE_+$ , for otherwise upon pull-back it would add ``tangential terms''
	 to $X_N$. Thus we have 
	  $\Phi (p,v) = (p, v + h(p, v)))$. Because $\Psi_1$ and $\Psi_2$ are both near-identity, so is 
	  $\Phi$ which means that when expanded as a power series $h(p,v)$ consists
	 entirely of terms quadratic and higher in $v$.  Write out its  convergent Taylor series 
	\[ h=h_2+h_3+\dots, \]
	with $ h_i (p, \cdot) : N_p \to N_p$ a homogeneous degree $ i $  vector-valued polynomial.  At the heart of the Poincar\'e-Dulac method is the fact that
	$\Phi^* X_N = X_N$ is equivalent to the (infinite)  system of ``cohomological equations''
	$ [L,h_i]= 0$.  See for instance \cite[Chapter~5]{Arnold} for details.   But the kernel of this cohomological operator
	$h \mapsto [h, L]$ is precisely  the linear span of the resonant monomials.  (To be precise, what  we mean
	by  a `monomial'   $h$ is any $h$ of  the form $h(v) = v^Q e_i$ with
	$e_i$  one of the  standard basis  elements for $\R^k$.)    We have seen that in our case there are no resonant monomials: this kernel is zero.
	Thus $h_i = 0$, $i =2,3,\ldots$ and hence  $\Phi = Id$ so that  $\Psi_1 = \Psi_2$ over $U_1 \cap U_2$. 
	Thus there is a single $\Psi: N - \to \V$,  defined over $\CE_+$,  whose restrictions to the various open
	sets  $U_p \subset \CE$ as per step 2  are the analytic near-identity diffeomorphisms  guaranteed by Brushlinskaya.
	
\end{proof}

\bibliographystyle{abbrv}

\bibliography{RefScattering}

\end{document}